\documentclass[a4paper,11pt
,parskip=half
,appendixprefix=true
,toc=bib
]{amsart}
\usepackage{amssymb,latexsym}
\usepackage{amsmath}
\usepackage{ifthen}
\usepackage{amsfonts}
\usepackage{amsthm}
\usepackage{enumerate}

\newtheorem{theorem}{Theorem}

\newtheorem{corollary}[theorem]{Corollary}

\newtheorem{definition}[theorem]{Definition}
\newtheorem{example}[theorem]{Example}

\newtheorem{lemma}[theorem]{Lemma}

\newtheorem{proposition}[theorem]{Proposition}
\newtheorem{remark}[theorem]{Remark}

\input{tcilatex}

\begin{document}
\title[Approximate Normality of Eigenfunctions]{Approximate Normality of
High-Energy Hyperspherical Eigenfunctions}
\author{Simon Campese}
\author{Domenico Marinucci}
\author{Maurizia Rossi}
\email{campese@mat.axp.uniroma2.it}
\email{marinucc@mat.axp.uniroma2.it}
\email{rossim@mat.axp.uniroma2.it}
\address{University of Rome ``Tor Vergata'', Via della Ricerca Scientifica
1, 00133 Roma, Italy}

\begin{abstract}
The Berry heuristic has been a long standing \emph{ansatz} about the high
energy (i.e. large eigenvalues) behaviour of eigenfunctions (see \cite{Berry
1977}). Roughly speaking, it states that under some generic boundary
conditions, these eigenfunctions exhibit Gaussian behaviour when the
eigenvalues grow to infinity. Our aim in this paper is to make this
statement quantitative and to establish some rigorous bounds on the distance
to Gaussianity, focussing on the hyperspherical case (i.e., for
eigenfunctions of the Laplace-Beltrami operator on the normalized $d$%
-dimensional sphere - also known as spherical harmonics). Some applications
to non-Gaussian models are also discussed.
\end{abstract}

\maketitle


\begin{itemize}
\item \textbf{Keywords and Phrases: }Hyperspherical Eigenfunctions, High
Energy Asymptotics, Geometry of Excursion Sets, $L^{\infty }$ norm,
Quantitative Central Limit Theorem.

\item \textbf{AMS Classification: }42C10; 60G60\smallskip ,60F05
\end{itemize}

\section{Introduction and Notation}

\label{s-sect-result-thro}

The Berry heuristic has been a long standing \emph{ansatz} about the high
energy (i.e. large eigenvalues) behaviour of eigenfunctions (see \cite{Berry
1977}). Roughly speaking, it states that under some generic boundary
conditions, these eigenfunctions exhibit Gaussian behaviour when the
eigenvalues grow to infinity (see below for more discussion and details).
Our aim in this paper is to make this statement quantitative and to
establish some rigorous bounds on the distance to Gaussianity, focussing on
the hyperspherical case (i.e., for eigenfunctions of the Laplace-Beltrami
operator on the normalized $d$-dimensional unit sphere - also known as
spherical harmonics).

The first of our formulations involves a wide class of geometric functionals
satisfying mild regularity conditions, including for example the excursion
area, the density of critical points above some threshold and the normalized
Euler-Poincar\'{e} characteristic of excursion sets (see below for details).
By the Berry heuristic, one should expect that for high energy, such
functionals, when evaluated at a typical eigenfunction, should be close to
the expectation of the same functionals evaluated in the corresponding
Gaussian case. To make this statement quantitative, we show that the
Lebesgue measure of eigenfunctions for which the difference of the two
aforementioned quantities lies above some vanishing threshold tends to zero
in the high-energy limit with some explicit rate. In the same spirit, we
prove that an analogous statement remains true if the geometric functionals
are replaced by supremum norms, again with quantitative bounds. In
particular, we show that the square of the supremum norms is typically of
logarithmic order; this latter finding complements some earlier
investigation on the $L^{4}$-norm of spherical eigenfunctions in dimension
two by Sogge and Zelditch (see~\cite{soggezelditch}).

As a second characterization, we focus on eigenfunctions which are evaluated
at a random point on the sphere and establish tight bounds on the
probability distances between these random variables and standard Gaussians.
In this setting, for dimension $d\geq 3$, we are furthermore able to
establish a form of almost sure convergence: consider sequences of
eigenfunctions containing a subsequence whose Kolmogorov distance to a
Gaussian stays above some vanishing threshold. Then the measure of such
sequences is zero (see below for precise statements). In this second
characterization, asymptotic Gaussianity is also exhibited if the eigenvalue
is kept fixed while the dimension of the underlying hypersphere grows to
infinity; this is related to the approach by Meckes (see~\cite{meckes}).

Finally, as an application, we investigate some non-Gaussian models.
Assuming some regularity on the sequence of probability measures (existence
of Lebesgue densities with suitable growth constraint), we show that the
asymptotic behaviour of functionals of excursion sets can be expressed in
terms of the Gaussian limits, evaluated at random excursion levels and
depending on the $L^{2}$-norm of sample paths. In particular, when this
random norm converges to a deterministic limit, Gaussian behaviour follows;
this way, we partially confirm an earlier conjecture from~\cite{marpecjmp}
on the relationship between convergence of sample norms and asymptotic
Gaussianity.

Let us now fix the mathematical framework. For a given dimension $d\geq 2$,
consider the orthonormal family 
\begin{equation*}
\left\{ \ Y_{\ell m}\mid \ell \in \mathbb{N}_{0},m=1,\dots ,n_{\ell
d}\right\} 
\end{equation*}%
of real hyperspherical harmonics on the normalized hypersphere $S^{d}$, i.e. 
\begin{equation*}
\Delta _{S^{d}}Y_{\ell m}=-\ell (\ell +d-1)Y_{\ell m},\qquad \ell \in 
\mathbb{N}_{0},\qquad m=1,\dots ,n_{\ell d},
\end{equation*}%
where 
\begin{equation*}
n_{\ell d}=\frac{2\ell +d-1}{\ell }\binom{\ell +d-2}{\ell -1}
\end{equation*}%
is the dimension of the eigenspace corresponding to the eigenvalue $-\ell
(\ell +d-1)$ of the Laplace-Beltrami operator $\Delta _{S^{d}}$ (see for
example \cite{andrews,atkinson,groemer,MaPeCUP}). Elementary computations
show that 
\begin{equation}
\lim_{\ell \rightarrow \infty }\frac{n_{\ell d}}{\ell ^{d-1}}=\frac{2}{(d-1)!%
}\qquad \text{ and }\qquad \lim_{d\rightarrow \infty }\frac{n_{\ell d}}{%
d^{\ell }}=\frac{1}{(\ell -1)!}.  \label{eq:4}
\end{equation}%
As anticipated above, we consider hyperspherical eigenfunctions~$h_{\alpha
,\ell }$ defined by 
\begin{equation}
h_{\alpha ,\ell }(x)=\sum_{m=1}^{n_{\ell d}}\alpha _{\ell m}Y_{\ell m}(x),
\label{eq:17}
\end{equation}%
where $\alpha _{\ell }=(\alpha _{\ell 1},\dots ,\alpha _{\ell n_{\ell
d}})\in S^{n_{\ell d}}$ is a vector of coefficients. Taking $\alpha _{\ell
}\in S^{n_{\ell d}}$ corresponds to the normalization 
\begin{equation*}
\left\Vert h_{\alpha ,\ell }\right\Vert _{L^{2}}^{2}=\int_{S^{d}}h_{\alpha
,\ell }^{2}(x)d\mu (x)=\sum_{m=1}^{n_{\ell d}}\alpha _{\ell m}^{2}=1,
\end{equation*}%
where $\mu $ represents the normalized Lebesgue measure on $S^{d}$. The link
to Gaussian eigenfunctions comes through the following natural and standard
randomization of the coefficient vector: Considering $(S^{n_{\ell d}},%
\mathcal{F}_{\ell },\mu _{\ell })$ as a probability space, where $\mathcal{F}%
_{\ell }$ is the Borel $\sigma $-algebra and $\mu _{\ell }$ the normalized
Lebesgue measure, we construct the probability space $(\Omega _{\ell },%
\mathcal{F}^{\ast },\mu _{\ell }^{\ast })$, where $\Omega _{\ell }=\mathbb{R}%
_{+}\times S^{n_{\ell d}}$, $\mathcal{F}_{\ell }^{\ast }=\mathcal{B}(\mathbb{%
R}_{+})\otimes \mathcal{F}_{\ell }$ and $\mu _{\ell }^{\ast }=\nu _{\ell
}\otimes \mu _{\ell }$. Here, $\nu _{\ell }$ is the measure induced by a
random radius $R_{\ell }$ defined as 
\begin{equation*}
R_{\ell }=\sqrt{\frac{X_{\ell ,d}}{n_{\ell d}}},
\end{equation*}%
where $X_{\ell ,d}\sim \chi ^{2}(n_{\ell d})$ is a random varioable
distributed as a chi-square with $n_{\ell d}$ degrees of freedom. It is
straightforward to check that under the measure $\mu _{\ell }^{\ast }$, the
random vector 
\begin{equation*}
u_{\ell }=(u_{\ell 1},\dots ,u_{\ell n_{\ell d}})\colon \Omega _{\ell
}\rightarrow \mathbb{R}^{n_{\ell d}},
\end{equation*}%
defined by $u_{\ell }(r,\alpha )=r\cdot \alpha $, have a multivariate
Gaussian distribution with covariance matrix $n_{\ell d}^{-1}I_{n_{\ell d}},$
$I_{n}$ denoting as usual the identity matrix of order $n$. We can hence
introduce the zero-mean, unit variance Gaussian eigenfunctions 
\begin{equation}
T_{\ell }(x)=T_{\ell }(x;r,\alpha ):=\sum_{m=1}^{n_{\ell d}}u_{\ell
m}(r,\alpha )Y_{\ell m}(x).  \label{eq:3}
\end{equation}%
Note that because we normalize the spherical measure to unity, we also have $%
\mathbb{E}\left\Vert T_{\ell }\right\Vert _{L^{2}(S^{d})}^{2}=\mathbb{E}%
[R_{\ell }^{2}]=1$, where, here and in the following, $\mathbb{E}$ denotes
mathematical expectation.

The rest of this paper is organized as follows: the main results are stated
in Section~\ref{s-main} and applied in Section~\ref{s-nongaussian} to some
non-Gaussian models. Most proofs are in Section \ref{s-proofs}, while we
have collected some technical lemmas in an appendix.

\section{Main results}

\label{s-main}

As stated in the introduction, we establish some quantitative result on
alternative versions of the Berry heuristics. Our first approach is to focus
on geometric functionals; we shall then consider the behaviour of supremum
norms, and finally eigenfunctions evaluated on random points.

\subsection{Excursion functionals}

\label{s-excursion} Let us first start from the standard definition of
excursion sets and monotonic functionals (see for example~\cite{adlertaylor}%
).

\begin{definition}[Excursion set]
\label{def:3} Let $f \colon S^{d}\rightarrow \mathbb{R}$ be a real valued
function; the excursion set of $f$ above $u\in \mathbb{R}$ is defined by%
\begin{equation*}
A_{u}(f;S^{d}):=\left\{ x\in S^{d}:f(x)\geq u\right\} \text{ .}
\end{equation*}
\end{definition}

\begin{definition}[Monotonic excursion functional]
\label{def:2} A functional $g\colon \mathcal{B}(S^{d})\rightarrow \mathbb{R}$
is called \emph{monotoni}c if either $A\subset B\Rightarrow g(A)\leq g(B)$
for all $A,B\in \mathcal{B}(S^{d})$ or $A\subset B\Rightarrow g(A)\geq g(B)$
for all $A,B\in \mathcal{B}(S^{d})$.
\end{definition}

\begin{remark}
\bigskip For excursion sets, we shall typically adopt the simpler notation%
\begin{equation*}
g(A_{u}(f;S^{d}))=:g(f,u)\text{ .}
\end{equation*}%
Note that, for every $c>0$, it holds that 
\begin{equation}
g(f,u)=g(cf,cu),  \label{quattro}
\end{equation}%
and of course $g$ is monotonic with respect to $u$, indeed either $u^{\prime
}\leq u\Rightarrow g(f,u^{\prime })\geq g(f,u)$ or $u^{\prime }\leq
u\Rightarrow g(f,u^{\prime })\leq g(f,u)$ for fixed $f\in L^{2}(S^{d})$.
\end{remark}

Examples of excursion functionals which have been studied in Gaussian
circumstances are excursion volumes (see \cite{MaWi12,mrossi}) and the
cardinality of critical points (see \cite{nicolaescu,cmw2014,cmw2015}); as
we shall show, other functionals such as the normalized Euler-Poincar\'{e}
characteristics satisfy more general conditions, but can be easily be
brought into the scope of our results below by simple manipulations. For our
results, we need to impose the following regularity property. As discussed
below, this condition is met by all the examples we mentioned.

\begin{definition}[regular family of monotonic excursion functionals]
\label{geometricfunctionals} Let $T_{\ell }$ be a Gaussian random field of
the form~\eqref{eq:3}. A family $\left\{ g_{\ell }\colon \ell \in \mathbb{N}%
\right\} $ of monotonic excursion functionals is called \emph{regular}, if
the following conditions are verified.

\begin{enumerate}
\item For all $u\in \mathbb{R}$ it holds that 
\begin{equation}
\mathbb{E}\left[ g_{\ell }(T_{\ell },u)\right] =\Psi _{\ell }(u)\rightarrow
\Psi (u),\qquad (\ell \rightarrow \infty ),  \label{eq:19}
\end{equation}%
where $\Psi _{\ell }$ is differentiable and, uniformly over $\ell ,$ 
\begin{equation}
\left\vert \Psi _{\ell }^{\prime }(u)\right\vert \leq \frac{c}{1+\left\vert
u\right\vert }  \label{eq:16}
\end{equation}%
for some $c>0$.

\item 
\begin{equation}
\sup_{u}\func{Var}\left( g_{\ell }(T_{\ell },u)\right) =o(\ell ).
\label{due}
\end{equation}
\end{enumerate}

\end{definition}

\begin{remark}
\label{rem:10} Note that the sequences of functionals $g_{\ell }$ can be
constant over $\ell $; we use this formulation for cases where we focus on
quantities whose expected values is bounded, as for instance for the
excursion volumes. On the other hand, to deal with diverging expected
values, as for instance the number of critical points above a threshold $u$
or the Euler-Poincar\'{e} characteristic (see below), it is convenient to
introduce a normalization depending on $\ell $.
\end{remark}

\begin{example}
Simple examples satisfying the previous conditions are obtained by
suppressing the dependence on $u$ and considering functionals of the form $%
\int_{S^{d}}M(T_{\ell })dx$, where $M$ is measurable such that $\mathbb{E}%
\left[ M(Z)^{2}\right] <\infty $, $Z$ being standard Gaussian. Indeed, under
these circumstances it was proved in~\cite{mrossi} that the variance is of
order $1/n_{\ell d}$. Likewise, we could consider functionals of the form $%
\int_{A_{u}(f;S^{d})}M(T_{\ell })dx$ where the function $M$ is even and
nonnegative: under these circumstances, it is indeed possible to show that
the variance is itself a monotonic (increasing) functional of the excursion
set.
\end{example}

Let us introduce some more notation: as before, let $\mu _{\ell }$ be the
normalized Lebesgue measure on $S^{n_{\ell d}}$; moreover, let $\left\{
g_{\ell }\colon \ell \in \mathbb{N}\right\} $ be a sequence of regular
excursion functionals. Set 
\begin{equation*}
\sigma _{\ell }^{2}(u)=\func{Var}\left( g_{\ell }(T_{\ell },u)\right)
,\qquad \sigma _{\ell }^{2}=\sup_{u}\sigma _{\ell }^{2}(u)
\end{equation*}%
and 
\begin{equation*}
G_{\ell }(\varepsilon ,u)=\left\{ \alpha _{\ell }\in S^{n_{\ell d}}\colon
\left\vert g_{\ell }(h_{\alpha ,\ell },u)-\mathbb{E}\left[ g_{\ell }(T_{\ell
},u)\right] \right\vert >\varepsilon \right\} ,
\end{equation*}%
where the eigenfunctions $h_{\alpha ,\lambda }$ and $T_{\ell }$ are defined
by~\eqref{eq:17} and~\eqref{eq:3}, respectively. A quantitative version of
the Berry heuristic for regular families of excursion functionals can now be
stated as follows.

\begin{theorem}
\label{thm:3} For all sequences $(\varepsilon _{\ell })_{\ell \geq 0}$ with
values in $(0,1)$ it holds that 
\begin{equation}
\sup_{u}\mu _{\ell }(G_{\ell }(\varepsilon _{\ell },u))\leq \frac{2(1+c)}{%
\varepsilon _{\ell }^{2}}\left( \frac{1}{n_{\ell d}}+\sigma _{\ell
}^{2}\right) ,  \label{eq:1}
\end{equation}%
where $c>0$ is such that the regularity condition (\ref{eq:16}) holds. In
particular, for vanishing sequences $\left\{ \varepsilon _{\ell }\right\}
_{\ell \geq 1}$ such that $1/n_{\ell d}+\sigma _{\ell }^{2}=o(\varepsilon
_{\ell }^{2})$, we have that 
\begin{equation*}
\sup_{u}\mu _{\ell }(G_{\ell }(\varepsilon _{\ell },u))\rightarrow 0,\qquad
(\ell \rightarrow \infty ).
\end{equation*}
\end{theorem}

Loosely speaking, this Theorem gives a bound on the Lebesgue measure of
eigenfunctions such that the corresponding geometric functionals diverge
from Gaussian behaviour by more than a (vanishing) sequence $\left\{
\varepsilon _{\ell }\right\} $.

\begin{remark}
A careful examination of the proof of Theorem~\ref{thm:3} shows that
locally, the improved bound 
\begin{equation*}
\mu _{\ell }(G_{\ell }(\varepsilon _{\ell },u))\leq \frac{2(1+c)}{%
\varepsilon _{\ell }^{2}}\left( \frac{1}{n_{\ell d}}+\max \left\{ \sigma
_{\ell }^{2}(u_{-}),\sigma _{\ell }^{2}(u_{+})\right\} \right)
\end{equation*}%
holds, where $u_{-}=\sqrt{1-\frac{\varepsilon }{1+c}}u$ and $u_{+}=\sqrt{1+%
\frac{\varepsilon }{1+c}}u$.
\end{remark}

\begin{example}
\label{ex:1} The following examples are easily seen to be covered by Theorem~%
\ref{thm:3}; in particular, for excursion volumes we can take the
functionals $\left\{ g_{\ell }\right\} $ to be constant with respect to $%
\ell ,$ while the remaining functionals require a normalization depending
upon $\ell .$ Technical details to show that the assumptions of the Theorem
hold are proved in the Appendix (see Lemmas~\ref{lemmaunob}, \ref{lem:2} and %
\ref{lemmaquattro}).

\begin{enumerate}
\item \emph{Excursion volume (see~\cite{mrossi,MaWi12,pham})} The excursion
volume $S(h_{\alpha ,\ell };u)$ is defined as 
\begin{equation}
S(h_{\alpha ,\ell };u)=\int_{S^{d}}1_{[u,\infty )}h_{\alpha ,\ell }(x))dx,
\label{eq:11}
\end{equation}%
where $1_{A}$ denotes as usual the indicator function of the set $A$. In
this case, the function $\Psi (u)=1-\Phi (u)$ is simply the Gaussian tail
probability, which is easily seen to satisfy the required regularity
conditions. Moreover, it is shown in Lemma \ref{lemmaunob} below (see also 
\cite{MaWi12,mrossi}) that, for all $d\geq 2$ 
\begin{equation*}
\func{Var}\left( S(T_{\ell };u)-[1-\Phi (u)]\right) =O\left( \frac{1}{%
n_{\ell d}}\right) ,
\end{equation*}%
uniformly over $u;$ hence \eqref{due} is also fulfilled. Thus, Theorem~\ref%
{thm:3} implies that 
\begin{equation*}
\mu _{\ell }\left\{ \alpha _{\ell }\colon \left\vert S(h_{\alpha ,\ell
};u)-1+\Phi (u)\right\vert >\varepsilon _{\ell }\right\} =O\left( \frac{1}{%
n_{\ell d}\varepsilon _{\ell }^{2}}\right) 
\end{equation*}%
and in particular, for all fixed $\varepsilon >0$%
\begin{equation*}
\mu _{\ell }\left\{ \alpha _{\ell }\colon \left\vert S(h_{\alpha ,\ell
};u)-1+\Phi (u)\right\vert >\varepsilon \right\} \rightarrow 0\ ,\text{ as }%
\ell \rightarrow \infty .
\end{equation*}

\item \emph{Normalized critical points, extrema and saddles (see~\cite%
{cmw2014}).} The critical points $N^{c}$, minima $N^{min}$, maxima $N^{max}$
and saddles $N^{s}$ are defined by 
\begin{align*}
N^{c}(h_{\alpha ,\ell };u)& =\#\left\{ x\in S^{d}\colon \nabla h_{\alpha
,\ell }(x)=0\text{ and }h_{\alpha ,\ell }(x)\geq u\right\} , \\
N^{min}(h_{\alpha ,\ell };u)& =\#\left( N^{c}(h_{\alpha ,\ell };u)\cap
\left\{ x\in S^{d}\colon \func{Ind}\left\{ \nabla ^{2}h_{\alpha ,\ell
}(x)\right\} =0\right\} \right) , \\
N^{s}(h_{\alpha ,\ell };u)& =N^{c}(h_{\alpha ,\ell };u)\setminus \left(
N^{min}(h_{\alpha ,\ell };u)\cup N^{max}(h_{\alpha ,\ell };u)\right)
\end{align*}%
where $\nabla $ and $\nabla ^{2}$ denote as usual the gradient and Hessian
on the sphere, and, for a matrix $A$, $\func{Ind}A$ denotes the number of
negative eigenvalues of $A$. Fixing $d=2$ it is shown in \cite{cmw2014} that 
\begin{equation*}
\lim_{\ell \rightarrow \infty }\frac{\mathbb{E}\left[ N^{b}(T_{\ell };u)%
\right] }{\ell ^{2}}=\Psi ^{b}(u)
\end{equation*}%
where $b\in \left\{ c,e,s\right\} $ for critical points, extrema and
saddles, respectively, and $\Psi ^{b}(u)=\int_{u}^{\infty }\psi ^{b}(z)dz$,
where 
\begin{align*}
\psi ^{c}(u)& =\frac{3}{\sqrt{2\pi }}(2e^{-t^{2}}+t^{2}-1)e^{-t^{2}}, \\
\psi ^{e}(u)& =\frac{3}{\sqrt{2\pi }}(e^{-t^{2}}+t^{2}-1)e^{-t^{2}}, \\
\psi ^{s}(u)& =\frac{3}{\sqrt{2\pi }}e^{-3t^{2}/2}.
\end{align*}%
The functions $\Psi ^{b}(u)$ are easily seen to satisfy the conditions
listed in Definition~\ref{geometricfunctionals}; in particular, it is shown
in \cite{cmw2014} that 
\begin{equation*}
\func{Var}\left( \frac{N^{b}(T_{\ell };u)}{{\ell }^{2}}-\Psi ^{b}(u)\right)
=O\left( \frac{1}{\ell }\right)
\end{equation*}%
uniformly over $u$ for $b\in \left\{ c,e,s\right\} $ (in fact, an analytic
expression for the leading constant is given, as a function of $u$). The
fact that identity~\eqref{quattro} is satisfied is shown in the appendix in
Lemma~\ref{lemmaquattro}. Hence, by Theorem~\ref{thm:3}, we have the
asymptotics 
\begin{equation*}
\mu _{\ell }\left\{ \alpha _{\ell }\colon \left\vert \frac{N^{b}(h_{\alpha
,\ell };u)}{{\ell }^{2}}-\Psi ^{b}(u)\right\vert >\varepsilon _{\ell
}\right\} =O\left( \frac{1}{\ell \varepsilon _{\ell }^{2}}\right) ,
\end{equation*}%
and again this measure converges to zero for any fixed $\varepsilon >0$.
\end{enumerate}
\end{example}

\begin{example}[Euler-Poincar\'{e} characteristic]
For a number of alternative definitions of the Euler-Poincar\'{e}
characteristic $\chi $ and its main properties, we refer to the monographs 
\cite{adlertaylor,adlerstflour}. For our purposes, it suffices to focus on
the two-dimensional case $d=2$ and excursion sets $A_{u}(h_{\alpha ,\ell })$
(see Definition~\ref{def:3}), where $\chi (A_{u}(h_{\alpha ,\ell }))$ can be
expressed by means of Morse's Theorem as the number of extrema minus the
number of saddles, i.e. 
\begin{equation*}
\chi (A_{u}(h_{\alpha ,\ell };S^{d}))=N^{e}(h_{\alpha ,\ell
};u)-N^{s}(h_{\alpha ,\ell };u).
\end{equation*}%
Equivalently, $\chi (A_{u}(h_{\alpha ,\ell };S^{d}))$ can be viewed as the
number of connected regions in $A_{u},$ minus the number of
\textquotedblleft holes\textquotedblright . Its expected value for Gaussian
random fields can be obtained by the celebrated Gaussian Kinematic Formula
(see for example \cite{adlertaylor,adlerstflour}). For $d=2$, we have that 
\begin{equation}
\mathbb{E}\left[ \chi (A_{u}(T_{\ell },S^{2}))\right] =2(1-\Phi (u))+\sqrt{%
\frac{2}{\pi }}\frac{\ell (\ell +1)}{2}\frac{u\phi (u)}{2}\text{ .}
\label{epc1}
\end{equation}%
The variance has been given in \cite{cmw2015} to be 
\begin{equation*}
\func{Var}\left( \chi (A_{u}(T_{\ell },S^{2}))\right) =\left(
(u^{3}+2u)^{2}\phi ^{2}(u)\right) ^{2}\frac{{\ell }^{3}}{8\pi }+O(\ell ^{2}).
\end{equation*}%
The function on the right-hand side of~\eqref{epc1} does not fulfill the
monotonicity conditions of Definition~\ref{def:2}. However, exploiting our
previous results we obtain%
\begin{align*}
\mu _{\ell }& \left\{ \alpha _{\ell }:\left\vert \frac{\chi (A_{u}(h_{\alpha
,\ell };S^{2}))}{\ell ^{2}}-\frac{1}{\sqrt{2\pi }}ue^{-u^{2}/2}\right\vert
\geq \varepsilon _{\ell }\right\} \\
& =\mu _{\ell }\left\{ \alpha _{\ell }:\left\vert \frac{\mathcal{N}%
^{e}(h_{\alpha ,\ell };u)-\mathcal{N}^{s}(h_{\alpha ,\ell }\ell ;u)}{\ell
^{2}}-\frac{1}{\sqrt{2\pi }}ue^{-u^{2}/2}\right\vert \geq \varepsilon _{\ell
}\right\} \\
& \leq \mu _{\ell }\left\{ \alpha _{\ell }:\left\vert \frac{\mathcal{N}%
^{e}(h_{\alpha ,\ell };u)}{\ell ^{2}}-\Psi ^{e}(u)\right\vert +\left\vert 
\frac{\mathcal{N}^{s}(h_{\alpha ,\ell };u)}{\ell ^{2}}-\Psi
^{s}(u)\right\vert \geq \varepsilon _{\ell }\right\} \\
& \leq \mu _{\ell }\left\{ \alpha _{\ell }:\left\vert \frac{\mathcal{N}%
^{e}(h_{\alpha ,\ell };u)}{\ell ^{2}}-\Psi ^{e}(u)\right\vert \geq \frac{%
\varepsilon _{\ell }}{2}\right\} \\
& \qquad \qquad +\mu _{\ell }\left\{ \alpha _{\ell }:\left\vert \frac{%
\mathcal{N}^{s}(h_{\alpha ,\ell };u)}{\ell ^{2}}-\Psi ^{s}(u)\right\vert
\geq \frac{\varepsilon _{\ell }}{2}\right\} \\
& =O\left( \frac{1}{\ell \varepsilon _{\ell }^{2}}\right) ,
\end{align*}%
where we have used the identity 
\begin{equation*}
\frac{1}{\sqrt{2\pi }}ue^{-u^{2}/2}=\frac{1}{\sqrt{2\pi }}\int_{u}^{\infty
}(t^{2}-1)e^{-t^{2}/2}dt=\Psi ^{e}(u)-\Psi ^{s}(u),
\end{equation*}%
and the first step can be found for example in \cite[eq.11.6.12, p.289]%
{adlertaylor}.
\end{example}

\subsection{The behaviour of $L^{\infty }$-norms}

\label{s-linfty} In this subsection, using the same approach as above, we
show that the set of hyperspherical eigenfunctions whose squared $L^{\infty
} $-norm is of logarithmic order has asymptotically measure one. To do so,
we first investigate two results on upper and lower bounds on the $L^{\infty
}$-norms in the Gaussian case, which are of some independent interest.

\begin{proposition}
\label{angela} For all $d\geq 2,$ the following is true.

\begin{enumerate}
\item There exists a constant $M>0$ such that 
\begin{equation}
\mathbb{E}\left[ \left\Vert T_{\ell }\right\Vert _{\infty }\right] \leq M%
\sqrt{\log \ell }  \label{eq:7}
\end{equation}%
for all $\ell \in \mathbb{N}$.

\item For $M$ satisfying~\eqref{eq:7} and $\beta >0$, it holds that 
\begin{equation*}
\mu _{\ell }^{\ast }\left\{ \left\Vert T_{\ell }\right\Vert _{\infty }\geq
\left( M+\sqrt{2\beta }\right) \sqrt{\log \ell }\right\} \leq \frac{1}{\ell
^{\beta }}.
\end{equation*}

\item For $0<K<\sqrt{d/(12d+2)}$ and $2K^{2}/d<\alpha <1/(6d+1)$ it holds
that%
\begin{equation*}
\mu _{\ell }^{\ast }\left\{ \left\Vert T_{\ell }\right\Vert _{\infty }\leq K%
\sqrt{\log \ell }\right\} =O\left( \frac{1}{\ell ^{(1-(6d+1)\alpha )/2}\log
\ell }\right) .
\end{equation*}
\end{enumerate}
\end{proposition}

Now, exploiting the same Gaussian approximation ideas as in the previous
subsection, we can establish the following quantitative results on the
behaviour of the supremum-norms for typical eigenfunctions.

\begin{theorem}
\label{inftynorm}\bigskip For $d\geq 2$, the following is true.

\begin{enumerate}
\item Let $M$ be such that~\eqref{eq:7} is verified. Then, for all constants 
$\beta,\beta^{\prime }$ such that $0<\beta^{\prime }<\beta$ it holds that 
\begin{equation*}
\mu _{\ell }\left\{ \alpha _{\ell } \colon \left\Vert h_{\alpha ,\ell
}\right\Vert_{\infty} \geq \left( M + \sqrt{2\beta} \right) \sqrt{\log \ell}
\right\} = O \left( \frac{1}{\ell^{\beta^{\prime }}} \right).
\end{equation*}

\item For $0 < K < \sqrt{d/(12d+2)}$ and $2K^2/d < \alpha < 1/(6d+1)$ it
holds that 
\begin{equation*}
\mu _{\ell }\left\{ \alpha _{\ell } \colon \left\Vert h_{\alpha ,\ell
}\right\Vert_{\infty} \leq K \sqrt{\log \ell} \right\} = O\left( \frac{1}{%
\ell^{(1-(6d+1)\alpha )/2} \log \ell }\right).
\end{equation*}
\end{enumerate}
\end{theorem}

A standard Borel-Cantelli argument immediately yields the following result
on the product space $\bigotimes_{\ell \geq 1} \Omega_{\ell}$.

\begin{corollary}
\label{cor:3} In the setting and with the notation of Theorem~\ref{inftynorm}%
, the set of eigenfunctions with fluctuations which are infinitely often
larger than $\left( M+\sqrt{2\beta }\right) \sqrt{\log \ell }$ is zero, i.e. 
\begin{equation}
\mu _{\infty }\left\{ (\alpha _{\ell })_{\ell \geq 1}\colon \left\Vert
h_{\alpha ,\ell }\right\Vert _{\infty }\geq (M+\sqrt{2\beta })\sqrt{\log
\ell }\text{ infinitely often}\right\} =0\text{,}  \label{infty2}
\end{equation}%
where $\mu _{\infty }$ denotes the product measure $\bigotimes_{\ell \geq
1}\mu _{\ell }$. Likewise, we have that 
\begin{equation}  \label{eq:10}
\mu_{\infty} \left\{ (\alpha_{\ell})_{\ell \geq 1} \colon \left\|
h_{\alpha,\ell} \right\|_{\infty} \leq K \sqrt{\log \ell} \text{ infinitely
often} \right\} = 0.
\end{equation}
\end{corollary}

\begin{remark}
\bigskip The behaviour of $L^{p}$ norms for spherical eigenfunctions has
been investigated by many authors. For eigenfunctions $e_{\lambda }$ of the
Laplace-Beltrami operator on a Riemannian manifold $M$ such that $\Delta
e_{\lambda }=-\lambda ^{2}e_{\lambda }$ and $\left\Vert e_{\lambda
}\right\Vert _{L^{2}(M)}=1$, Sogge~\cite{soggejfa88} obtained the
asymptotics 
\begin{equation*}
\left\Vert e_{\lambda }\right\Vert _{L^{p}(M)}=O(\lambda ^{\sigma
(p)}),\qquad 2<p\leq \infty ,
\end{equation*}%
where 
\begin{equation*}
\sigma (p)=%
\begin{cases}
2(\frac{1}{2}-\frac{1}{p})-\frac{1}{2} & \qquad \text{ if $p\geq 6$,} \\ 
\frac{1}{2}(\frac{1}{2}-\frac{1}{p}) & \qquad \text{ if $2<p\leq 6$}.%
\end{cases}%
\end{equation*}%
On $S^{2}$, this implies in particular that $\left\Vert Y_{\ell
m}\right\Vert _{L^{4}(S^{2}\mathcal{)}}=O(\ell ^{1/8})$ and $\left\Vert
Y_{\ell m}\right\Vert _{L^{\infty }(S^{2}\mathcal{)}}=O(\ell ^{1/2});$ these
estimates are known to be sharp, with the two limits achieved by $Y_{\ell
\ell }$ and $Y_{\ell 0},$ respectively. However, for $d=2$ it is also known
that the typical eigenfunction has much smaller norms; in particular, it has
been shown more recently by Sogge and Zelditch in~\cite{soggezelditch} that
the average $L^{4}$ norm is logarithmic, more precisely%
\begin{equation*}
\frac{1}{2\ell +1}\sum_{m=-\ell }^{\ell }\int_{S^{2}}\left\vert Y_{\ell
m}(x)\right\vert ^{4}dx=O(\log \ell )\text{ , as }\ell \rightarrow \infty 
\text{ .}
\end{equation*}
\end{remark}

\begin{remark}
The behaviour of $L^{\infty }$ norms for spherical harmonics plays a crucial
role in the analysis of the recovery properties for sparse spherical signals
by compressed sensing techniques, see for instance \cite{rauhutward0},\cite%
{rauhutward}. The results we derived here may lead to improved bounds on
reconstruction errors - we leave this as an issue for future research.
\end{remark}

\subsection{Random evaluation of eigenfunctions}

\label{s-random} The previous subsections were concerned with the path
behaviour of eigenfunctions, showing that the set where the latter differ
from Gaussian behaviour is asymptotically negligible. In this section, we
address a slightly different question, namely we investigate the behaviour
of typical eigenfunctions when evaluated at a random point on the sphere;
this is the same setting as considered in \cite{meckes}. Let $X\sim U(S^{d})$
be uniformly distributed on the $d$-dimensional sphere. In our notation, it
is proved in~\cite{meckes} that for any fixed $\alpha \in S^{n_{\ell ,d}}$,
the random variable $h_{\alpha ,\ell }(X)=\sum_{m}\alpha _{\ell m}Y_{\ell
m}(X)$ is such that 
\begin{equation*}
d_{TV}(h_{\alpha ,\ell }(X),Z)\leq \frac{2}{\ell (\ell +d-1)}%
\int_{S^{d}}\left\vert \left\Vert \nabla h_{\alpha ,\ell }(x)\right\Vert _{%
\mathbb{R}^{d}}^{2}-\left( \int_{S^{d}}\left\Vert \nabla h_{\alpha ,\ell
}(y)\right\Vert _{\mathbb{R}^{d}}dy\right) ^{2}\right\vert d\,x
\end{equation*}%
where $d_{TV}$ denotes as usual the total variation distance between random
variables, i.e.%
\begin{equation*}
d_{TV}(h_{\alpha ,\ell }(X),Z)=\sup_{A\in \mathcal{B}(\mathbb{R)}}\left\vert
\int_{S^{d}}1_{A}\left( h_{\alpha ,\ell }(x))dx\right) dx-\int_{A}\phi
(u)du\right\vert .
\end{equation*}%
In the sequel, we focus instead on the (weaker) Kolmogorov distance, which
is defined by 
\begin{equation*}
d_{Kol}(h_{\alpha ,\ell }(X),Z)=\sup_{u}\left\vert \int_{S^{d}}1_{[u,\infty
)}\left( h_{\alpha ,\ell }(x)\right) dx-\Phi (u)\right\vert ,
\end{equation*}%
(see \cite{goldchen,noupebook} for more discussion on probability metrics).
Our next result establishes the rate of convergence to asymptotic
Gaussianity in the Kolmogorov distance for the typical eigenfunctions
evaluated at a random point.

\begin{theorem}
\label{prop:1} There exists a universal constant $K$ such that for fixed $d$
and for all positive sequences $\varepsilon _{\ell ,d}$ it holds that 
\begin{equation}
\mu _{\ell }\left( \left\{ \alpha _{\ell }\colon d_{\text{Kol}}(h_{\alpha
,\ell ,d}(X),Z)>\varepsilon _{\ell ,d}\right\} \right) \leq \frac{K}{n_{\ell
d}\,\varepsilon _{\ell ,d}^{3}}=O\left( \frac{1}{{\ell }^{d-1}\varepsilon
_{\ell ,d}^{3}}\right)   \label{eq:5}
\end{equation}%
as $\ell \rightarrow \infty $. Likewise, there exists a universal constant $%
K^{\prime }$ such that for fixed $\ell $ and for all positive sequences $%
\varepsilon _{\ell ,d}$ it holds that 
\begin{equation}
\mu _{\ell }\left( \left\{ \alpha _{\ell ,d}\colon d_{\text{Kol}}(h_{\alpha
,\ell ,d}(X),Z)>\varepsilon _{\ell ,d}\right\} \right) \leq \frac{K}{n_{\ell
d}\,\varepsilon _{\ell ,d}^{3}}=O\left( \frac{1}{{d}^{\ell }\varepsilon
_{\ell ,d}^{3}}\right)   \label{eq:5}
\end{equation}%
as $d\rightarrow \infty $.
\end{theorem}

Clearly, the previous proposition implies that for sequences $\varepsilon
_{\ell ,d}=\frac{s_{\ell ,d}}{n_{\ell d}^{1/3}}$ where $s_{\ell
,d}\rightarrow \infty $ it holds that 
\begin{equation*}
\mu _{\ell }\left( \left\{ \alpha _{\ell ,d}\colon d_{\text{Kol}}(h_{\alpha
,\ell ,d}(X),Z)>\varepsilon _{\ell ,d}\right\} \right) \rightarrow 0
\end{equation*}%
as $\ell \rightarrow \infty $. It is important to stress that the second
statement in the previous Theorem implies that asymptotic Gaussianity also
holds for fixed $\ell $ and growing dimension $d\rightarrow \infty $. This
is a setting closer to the assumptions of \cite{meckes}; in particular, in
Theorem 10 of that reference, it is proved that (again in our notation) 
\begin{equation}
\mathbb{E}\left[ d_{TV}(h_{\alpha ,\ell ,d}(X),Z)\right] \leq \frac{K}{\sqrt{%
d}},\text{ some }K>0,
\end{equation}%
where the expected value is taken with respect to the measure $\mu _{\ell }$%
, i.e. the uniform distribution of the coefficients $\alpha _{\ell ,d}$.
Working with Kolmogorov distance rather than Total Variation, we can
strengthen the bound as follows.

\begin{corollary}
\label{cor:1} In the setting of Theorem~\ref{prop:1} and for fixed $\ell $
it holds that 
\begin{equation*}
\mathbb{E}\left[ d_{Kol}(h_{\alpha ,\ell ,d}(X),Z)\right] =O\left( \frac{1}{%
d^{\ell /3}}\right)
\end{equation*}%
as $d\rightarrow \infty $. Also, for fixed $d$, it holds that 
\begin{equation*}
\mathbb{E}\left[ d_{Kol}(h_{\alpha ,\ell ,d}(X),Z)\right] =O\left( \frac{1}{%
\ell ^{(d-1)/3}}\right)
\end{equation*}%
as $\ell \rightarrow \infty $.
\end{corollary}

For $d$ large enough, the sequence $\frac{1}{n_{\ell d}\varepsilon _{\ell
,d}^{1/3}}$ can be chosen to be summable over $\ell $ and the results we
established can be formulated in terms of almost sure convergence by
standard Borel-Cantelli arguments. The precise result is as follows.

\begin{proposition}
\label{prop:2} Let $(\Omega ,\mathcal{F},\mu )$ be the product probability
space of the spaces $(\Omega _{\ell },\mathcal{F}_{\ell },\mu _{\ell })$.
For all sequences $\left\{ n_{\ell d}\right\} $ and $\left\{ \varepsilon
_{\ell }>0\right\} $ such that 
\begin{equation*}
\sum_{\ell =0}^{\infty }\frac{1}{n_{\ell d}\varepsilon _{\ell }^{3}}<\infty
\end{equation*}%
it holds that 
\begin{equation*}
\mu \left( \left\{ \,(\alpha _{\ell })_{\ell =1}^{\infty }\colon
d_{Kol}(h_{\alpha ,\ell }(X),Z)>\varepsilon _{\ell }\text{ for infinitely
many $\ell $ }\right\} \right) =0
\end{equation*}
\end{proposition}

\begin{example}
For $d=3$, $n_{\ell d}$ is of order $\ell ^{2}$ and the previous almost sure
convergence result holds for all sequences $\varepsilon _{\ell }$ such that $%
1/\varepsilon _{\ell }=O({\ell }^{1/3-\delta })$ for some $\delta >0$. More
generally, for $d\geq 3$, convergence will hold provided that $\varepsilon
_{\ell }$ is of the form ${\ell }^{(2-d)/3+\delta }$ for some $\delta >0$.
Again, this shows that the Gaussian approximation becomes more and more
accurate in higher dimension.
\end{example}

\section{Applications: Some non-Gaussian Models}

\label{s-nongaussian}

The literature on geometric properties of random eigenfunctions has so far
focused only on the case where the latter are normally distributed; in this
section, we shall show how one can exploit the previous results to
investigate the asymptotic behaviour of geometric functionals under some
non-Gaussian circumstances. To do so, let us consider the eigenfunctions 
\begin{equation}
\widetilde{T}_{\ell }(x)=\sum_{m=1}^{n_{\ell d}}\widetilde{u}_{\ell
m}Y_{\ell m}(x),  \label{eq:9}
\end{equation}%
defined analogously to~\eqref{eq:4}, but now allowing non-Gaussian
distributions for the sequence of coefficient vectors $\widetilde{u}_{\ell
}=(\widetilde{u}_{\ell 1},\dots ,\widetilde{u}_{\ell n_{ld}})$, $\widetilde{u%
}_{\ell }\colon \Omega _{\ell }\rightarrow \mathbb{R}^{d}$, distributed
according to the (non-Gaussian) sequence of measures $(\widetilde{\mu }%
_{\ell })$. We take $\widetilde{\mu }_{\ell }$ to be absolutely continuous
with respect to the Lebesgue measure $\mu _{\ell }$ on the unit sphere; in
other words, we assume that the vector $\widetilde{u}_{\ell }$ admits a
probability density, for all $\ell .$

Our first result in this section is a straightforward extension of Theorem~%
\ref{thm:3}.

\begin{proposition}
\label{thm:1} In the setting of Theorem~\ref{thm:3}, denote by $p_{\ell }:=%
\frac{\mathrm{d}\mu _{\ell }}{\mathrm{d}\widetilde{\mu }_{\ell }}$ the
sequence of Radon-Nikodym derivatives and let $(g_{\ell })_{\ell \geq 0}$ be
a regular family of excursion functionals. Then for all positive sequences $%
(\varepsilon _{\ell })_{\ell \geq 0}$ it holds that%
\begin{equation*}
\Pr \left( \sup_{u}\left\vert g_{\ell }(\widetilde{T}_{\ell };u)-\Psi
(u)\right\vert >\varepsilon _{\ell }\right) =O\left( \left\Vert p_{\ell
}\right\Vert _{\infty }\left( \frac{1}{\sigma _{\ell }^{2}\varepsilon _{\ell
}^{3}}+\frac{1}{n_{\ell ;d}\varepsilon _{\ell }^{2}}\right) \right) \text{ ,}
\end{equation*}%
as $\ell \rightarrow \infty $.
\end{proposition}

\begin{remark}
Of course, Proposition~\ref{thm:1} only yields interesting results if~$%
\left\| p_{\ell} \right\|_{\infty}$ grows slower than~$\sigma_{\ell}^2 +
n_{\ell d}$.
\end{remark}

\begin{proof}
Because each $g_{\ell }$ is by assumption monotonic in the second variable,
the proof that%
\begin{equation*}
\mu _{\ell }\left( \left\{ \alpha _{\ell }:\sup_{u}\left\vert g_{\ell
}(h_{\alpha _{\ell },d};u)-\Psi (u)\right\vert >\varepsilon _{\ell }\right\}
\right) =O\left( \frac{1}{\sigma _{\ell }^{2}\varepsilon _{\ell }^{3}}+\frac{%
1}{n_{\ell ;d}\varepsilon _{\ell }^{2}}\right)
\end{equation*}%
can be given exactly as in the proof of Theorem \ref{prop:1}, i.e. by
showing first that%
\begin{equation*}
\mu _{\ell }^{\ast }\left( \left\{ (r_{\ell },\alpha _{\ell })\colon
\sup_{u}\left\vert g_{\ell }(T_{\ell };u)-\Psi (u)\right\vert >\varepsilon
_{\ell }\right\} \right) =O\left( \frac{1}{\sigma _{\ell }^{2}\varepsilon
_{\ell }^{3}}\right)
\end{equation*}%
and then exploiting the fact that%
\begin{equation*}
\mu _{\ell }^{\ast }\left( \left\{ \left\vert R_{\ell }-1\right\vert
>\varepsilon _{\ell }\right\} \right) =O\left( \frac{1}{n_{\ell
;d}\varepsilon _{\ell }^{2}}\right) .
\end{equation*}%
To conclude the argument, it then suffices to define as before 
\begin{equation*}
\widetilde{G}_{\ell }(\varepsilon )=\left\{ \widetilde{u}_{\ell }\in
S^{n_{\ell d}}\colon \sup_{u}\left\vert g_{\ell }(\widetilde{T}_{\ell
};u)-\Psi (u)\right\vert >\varepsilon _{\ell }\right\}
\end{equation*}%
and to note that%
\begin{equation*}
\Pr \left( \sup_{u}\left\vert g_{\ell }(\widetilde{T}_{\ell };u)-\Psi
(u)\right\vert >\varepsilon _{\ell }\right)
\end{equation*}%
\begin{eqnarray*}
\widetilde{\mu }_{\ell }(\widetilde{G}_{\ell }(\varepsilon )) &=&\int_{%
\widetilde{G}_{\ell }(\varepsilon )}\func{d}\widetilde{\mu }_{\ell
}=\int_{\left\{ \alpha _{\ell }:\sup_{u}\left\vert g_{\ell }(h_{\alpha
_{\ell },d};u)-\Psi (u)\right\vert >\varepsilon _{\ell }\right\} }p_{\ell }\,%
\mathrm{d}\mu _{\ell } \\
&\leq &\left\Vert p_{\ell }\right\Vert _{\infty }\mu _{\ell }\left( \left\{
\alpha _{\ell }:\sup_{u}\left\vert g_{\ell }(h_{\alpha _{\ell },d};u)-\Psi
(u)\right\vert >\varepsilon _{\ell }\right\} \right) \\
&=&O\left( \frac{1}{\sigma _{\ell }^{2}\varepsilon _{\ell }^{3}}+\frac{1}{%
n_{\ell ;d}\varepsilon _{\ell }^{2}}\right)
\end{eqnarray*}%
as claimed.
\end{proof}

\begin{remark}
It should be noticed that under isotropy, the random coefficients of any
non-Gaussian model must satisfy 
\begin{equation*}
D_{\ell }(g) \widetilde{u}_{\ell .}\overset{d}{=} \widetilde{u}_{\ell .}%
\text{ for all }g\in SO(d+1),
\end{equation*}
where $D_{\ell}$ denotes the $\ell$th irreducible representation of $SO(d+1)$
(for $d=2$ the set $\left\{ D_{\ell}(g) \colon g \in SO(3) \right\}$ is the
well known family of $(2\ell+1)\times(2\ell + 1)$ unitary Wigner matrices).
In the Gaussian case, this identity in distribution is actually obvious,
because the distribution of the vector $\widetilde{u}_{\ell}$ is uniform on
a sphere of random radius. Our result in this section heuristically suggests
that, under regularity conditions, the distribution of the vector of random
coefficients should be close to uniform as well in the high energy limit. We
leave it for further research to relate this behaviour with the mixing
properties of the matrices $D_{\ell}(g)$.
\end{remark}

A simple example of non-Gaussian eigenfunctions is provided by 
\begin{equation*}
\widetilde{T}_{\ell }(x)=\xi _{\ell }T_{\ell }(x), \qquad \ell \in \mathbb{N}%
,
\end{equation*}%
where the sequence $(T_{\ell })$ is as before Gaussian and isotropic with
zero mean and unit variance, while $(\xi _{\ell })$ is an arbitrary sequence
of random variables. Of course, in this case one trivially obtains a
Gaussian limiting behaviour normalizing the sequence $(\widetilde{T}_{\ell
}) $ by its (random) $L^{2}$-norm~$\left\vert \xi _{\ell }\right\vert
\left\Vert T_{\ell }\right\Vert _{L^{2}}.$ It is natural to ask whether such
a result could hold under more general circumstances and the following
simple Corollary provides a partial positive answer.

\begin{corollary}
\label{cor:2} Let $(\widetilde{T}_{\ell })$ be a sequence of isotropic
non-Gaussian eigenfunctions of the form~\eqref{eq:9} such that the density
sequence of the vectors $\left( \widetilde{u}_{\ell }/\sqrt{\widetilde{C}%
_{\ell }}\right) $ of normalized spherical harmonic coefficients is
uniformly bounded as $\ell \rightarrow \infty $, where 
\begin{equation}
\widetilde{C}_{\ell }=\sum_{m=1}^{n_{\ell d}}\widetilde{u}_{\ell m}^{2}
\label{eq:15}
\end{equation}%
denotes the sample angular power spectrum. Then, for a regular family $%
(g_{\ell })$ of monotonic excursion functionals it holds that 
\begin{equation}
\Pr \left( \sup_{u\in \mathbb{R}}\left\vert g_{\ell }(\widetilde{T}_{\ell
};u)-\Psi \left( \frac{u}{\sqrt{\widetilde{C}_{\ell }}}\right) \right\vert
>\varepsilon _{\ell }\right) =O\left( \frac{1}{\sigma _{\ell
}^{2}\varepsilon _{\ell }^{3}}+\frac{1}{n_{\ell ;d}\varepsilon _{\ell }^{2}}%
\right)  \label{eq:8}
\end{equation}%
as $\ell \rightarrow \infty $.
\end{corollary}

\begin{proof}
We have 
\begin{align*}
\sup_{u\in \mathbb{R}}\left\vert g_{\ell }(\widetilde{T}_{\ell };u)-\Psi
\left( \frac{u}{\sqrt{\widetilde{C}_{\ell }}}\right) \right\vert &
=\sup_{u\in \mathbb{R}}\left\vert g_{\ell }\left( \frac{\widetilde{T}_{\ell }%
}{\sqrt{\widetilde{C}_{\ell }}};\frac{u}{\sqrt{\widetilde{C}_{\ell }}}%
\right) -\Psi \left( \frac{u}{\sqrt{\widetilde{C}_{\ell }}}\right)
\right\vert  \\
& \leq \sup_{v\in \mathbb{R}}\left\vert g_{\ell }\left( \frac{\widetilde{T}%
_{\ell }}{\sqrt{\widetilde{C}_{\ell }}};v\right) -\Psi \left( v\right)
\right\vert .
\end{align*}%
Thus, the result follows from Proposition \ref{thm:1}.
\end{proof}

In words, Corollary~\ref{cor:2} is stating that under regularity
assumptions, and conditionally on the value of the norm of its random
coefficients, the geometric functionals evaluated at non-Gaussian
eigenfunctions converge to a rescaled version of the Gaussian limit - or,
equivalently, to the Gaussian limit evaluated at a random point depending on
the norm of the coefficients. Let us now reconsider the previous examples.

\begin{example}[Excursion volume]
\emph{(See \cite{mrossi,MaWi12,pham}) }\label{ex:3} Using the notation
introduced in Example~\ref{ex:1}, we get that 
\begin{equation*}
\sup_{u\in \mathbb{R}}\left\vert S(\widetilde{T_{\ell }},u)-\left( 1-\Phi
\left( \frac{u}{\sqrt{\widetilde{C}_{\ell }}}\right) \right) \right\vert
\rightarrow _{p}0,
\end{equation*}%
where $\rightarrow _{p}$ denotes convergence in probability. Assuming for
instance that $\widetilde{C}_{\ell }$ converges in distribution to some
limiting random variable $C_{\infty }$, we have that $S(\widetilde{T}_{\ell
};u)$ converges in distribution to $1-\Phi \left( \frac{u}{\sqrt{C_{\infty }}%
}\right) $. In particular, the Gaussian limiting behaviour is obtained once
more if and only if $C_{\infty }=1$ with probability one.
\end{example}

\begin{example}[Critical points]
\emph{(See~\cite{cmw2014}) }\label{ex:2} Recall the Definition of $N^{c}$
and $\Psi ^{c}$ from Example~\ref{ex:1}.2. Under the assumptions of
Corollary~\ref{cor:2}, exactly the same argument as before yields 
\begin{equation*}
\sup_{u\in \mathbb{R}}\left\vert \frac{N^{c}(\widetilde{T}_{\ell };u)}{\ell
^{2}}-\Psi ^{c}\left( \frac{u}{\sqrt{\widetilde{C}_{\ell }}}\right)
\right\vert \rightarrow _{p}0
\end{equation*}%
as $\ell \rightarrow \infty $. As before, assuming $\widetilde{C}_{\ell }$
converges in distribution to some limiting random variable $C_{\infty }$, we
have 
\begin{equation*}
\frac{N^{c}(\widetilde{T}_{\ell };u)}{\ell ^{2}}\rightarrow _{d}\Psi
^{c}\left( \frac{u}{\sqrt{C_{\infty }}}\right) .
\end{equation*}
\end{example}

Examples~\ref{ex:3}-\ref{ex:2} are basically stating that the limiting
behaviour in these non-Gaussian circumstances corresponds to a mixture of
the Gaussian limiting expressions with a (random) scaling factor.

As a final remark, we note that the previous results suggest that
convergence of the random norm of the spherical harmonic coefficients to a
constant may be closely related to the asymptotic Gaussianity of
hyperspherical eigenfunctions. This statement is made rigorous in the
following corollary; we recall the standard notation $X_{n}=O_{p}(d_{n})$ to
denote that the sequence $\left\vert \frac{X_{n}}{d_{n}}\right\vert $ is
bounded in probability, i.e., for all $\varepsilon >0$ there exist $n_{0}\in 
\mathbb{N}$ and $K>0$ such that $\func{P} \left( \left\vert \frac{X_{n}}{%
d_{n}}\right\vert >K\right) <\varepsilon $ for all $n>n_{0}.$

\begin{corollary}
Let $\widetilde{T}_{\ell }(x)$ be given as in~\eqref{eq:9}, and assume
moreover that 
\begin{equation*}
\sum_{m=1}^{n_{\ell d}}\widetilde{u}_{\ell m}^{2}-1 =O_{p}(\gamma _{\ell })%
\text{ , }\gamma _{\ell }\rightarrow 0\text{ as }\ell \rightarrow \infty .
\end{equation*}%
Then, under the assumptions of Corollary \ref{cor:2}, we have that 
\begin{equation*}
d_{Kol}(\widetilde{T}_{\ell }(x),Z)=O \left(\frac{1}{\sigma _{\ell
}^{2}\varepsilon _{\ell }^{3}}+\frac{1}{n_{\ell ;d}\varepsilon _{\ell }^{2}}%
+\gamma _{\ell }\right)
\end{equation*}
for all $x\in S^{d}$, where $Z$ is standard Gaussian.
\end{corollary}

\begin{proof}
In view of Corollary \ref{cor:2} and the isotropy of $T$, we know that%
\begin{eqnarray*}
\Pr \left( \frac{\widetilde{T}_{\ell }(x)}{\sqrt{\widetilde{C}_{l}}}\leq
u\right) -\Phi (u) &=&\mathbb{E}\left[ \ 1_{(-\infty ,u]}\left( \frac{%
\widetilde{T}_{\ell }(x)}{\sqrt{\widetilde{C}_{l}}}\right) \right] -\Phi (u)
\\
&=&\mathbb{E}\left[ \ \int_{S^{d}}\ 1_{(-\infty ,u]}\left( \frac{\widetilde{T%
}_{\ell }(x)}{\sqrt{\widetilde{C}_{l}}}\right) \mathrm{d}x-\Phi (u)\right] \\
&=&O\left( \frac{1}{\sigma _{\ell }^{2}\varepsilon _{\ell }^{3}}+\frac{1}{%
n_{\ell ;d}\varepsilon _{\ell }^{2}}\right) ,
\end{eqnarray*}%
uniformly in $u$. Therefore, as by assumption $\sqrt{\widetilde{C}_{\ell }}%
\rightarrow _{p}1$, it holds that 
\begin{equation*}
\sup_{u}\left\vert \Pr \left( \widetilde{T}_{\ell }(x)\leq u\right) -\Phi
(u)\right\vert =\sup_{u}\left\vert \Pr \left( \frac{\widetilde{T}_{\ell }(x)%
}{\sqrt{\widetilde{C}_{{\ell }}}}\leq \frac{u}{\sqrt{\widetilde{C}_{{\ell }}}%
}\right) -\Phi (u)\right\vert
\end{equation*}%
\begin{equation*}
\leq \sup_{u}\left\vert \Pr \left( \frac{\widetilde{T}_{\ell }(x)}{\sqrt{%
\widetilde{C}_{{\ell }}}}\leq \frac{u}{\sqrt{\widetilde{C}_{{\ell }}}}%
\right) -\Phi \left( \frac{u}{\sqrt{\widetilde{C}_{{\ell }}}}\right)
\right\vert +\sup_{u}\left\vert \Phi \left( \frac{u}{\sqrt{\widetilde{C}_{{%
\ell }}}}\right) -\Phi (u)\right\vert \rightarrow 0,
\end{equation*}%
as $\ell \rightarrow \infty .$ To conclude the proof, it suffices to notice
that, by a simple application of the Mean Value Theorem, 
\begin{eqnarray*}
\sup_{u}\left\vert \Phi \left( \frac{u}{\sqrt{\widetilde{C}_{{\ell }}}}%
\right) -\Phi (u)\right\vert &\leq &\sup_{u}\left\vert \Phi (u)-\Phi \left( 
\sqrt{\widetilde{C}_{{\ell }}}u\right) \right\vert \\
&\leq &\sup_{u}\left\vert \phi (u)\right\vert \left\vert u\right\vert
\left\vert \sqrt{\widetilde{C}_{{\ell }}}-1\right\vert \\
&\leq &\frac{1}{\sqrt{2\pi \mathrm{e}}\left( \sqrt{\widetilde{C}_{\ell }}%
+1\right) }\,\left\vert \widetilde{C}_{{\ell }}-1\right\vert \\
&=&O_{p}(\gamma _{\ell }).
\end{eqnarray*}

\end{proof}

Of course, if $\widetilde{T}_{\ell }(x)$ is Gaussian, the reverse
implication is obvious by the standard law of large numbers (note that under
isotropy the $\widetilde{u}_{\ell m}$ are necessarily uncorrelated , while
they are independent if and only if the field is Gaussian; see for instance 
\cite{baldimarinucci,balditrapani}, or \cite{MaPeCUP}, Chapter 5). The
previous corollary hence partially confirms a conjecture on the relationship
between high frequency ergodicity and high frequency Gaussianity that was
raised a few years ago by \cite{marpecjmp}.

\section{Proofs \label{s-proofs}}

We begin with the proof of Theorem~\ref{thm:3} from Section~\ref{s-excursion}%
, which is rather straightforward; the details are as follows.

\begin{proof}[Proof of Theorem~\protect\ref{thm:3}]
Recall the construction of the Gaussian measure $\mu _{\ell }^{\ast }=\nu
_{\ell }\otimes \mu _{\ell }$ from the introduction, obtained by adjoining a
random radius distributed as $\sqrt{\frac{X_{\ell ,d}}{n_{\ell d}}},$ where $%
X_{\ell ,d}\sim \chi _{n_{\ell d}}^{2},$ to the normalized Lebesgue measure $%
\mu _{\ell }$ on $S^{n_{\ell d}-1}$. Define $\Psi _{\ell }(u)=\mathbb{E}%
\left[ g_{\ell }(T_{\ell },u)\right] $. It holds that 
\begin{align*}
\mu _{\ell }(G_{\ell }(\varepsilon ,u))& =\mu _{\ell }\left\{ \alpha _{\ell
}\colon \left\vert g_{\ell }(h_{\alpha ,\ell },u)-\Psi _{\ell
}(u)\right\vert \geq \varepsilon \right\}  \\
& =\mu _{\ell }^{\ast }\left\{ (r_{\ell },\alpha _{\ell })\colon \left\vert
g_{\ell }(h_{\alpha ,\ell },u)-\Psi _{\ell }(u)\right\vert \geq \varepsilon
\right\} ,
\end{align*}%
where, here and in the following, we tacitly assume that $r_{\ell }\in 
\mathbb{R}_{+}$ and $\alpha _{\ell }\in S^{n_{\ell d}-1}$. By the law of
total probability, the above can be bounded by 
\begin{align}
\nu _{\ell }& \left\{ r_{\ell }\colon \left\vert r_{\ell }^{2}-1\right\vert
\geq \lambda \right\}   \notag \\
& \qquad \qquad +\mu _{\ell }^{\ast }\left\{ (r_{\ell },\alpha _{\ell
})\colon \left\vert g_{\ell }(h_{\alpha ,\ell },u)-\Psi _{\ell
}(u)\right\vert \geq \varepsilon ,\,\left\vert r_{\ell }^{2}-1\right\vert
<\lambda \right\}   \notag \\
& =\nu _{\ell }\left\{ r_{\ell }\colon \left\vert r_{\ell }^{2}-1\right\vert
\geq \lambda \right\}   \notag \\
& \qquad \qquad +\mu _{\ell }^{\ast }\left\{ (r_{\ell },\alpha _{\ell
})\colon g_{\ell }(h_{\alpha ,\ell },u)-\Psi _{\ell }(u)\geq \varepsilon
,\,\left\vert r_{\ell }^{2}-1\right\vert <\lambda \right\}   \label{eq:18} \\
& \qquad \qquad +\mu _{\ell }^{\ast }\left\{ (r_{\ell },\alpha _{\ell
})\colon g_{\ell }(h_{\alpha ,\ell },u)-\Psi _{\ell }(u)\leq -\varepsilon
,\,\left\vert r_{\ell }^{2}-1\right\vert <\lambda \right\}   \notag
\end{align}%
For the first measure on the right hand side, Chebyshev's inequality yields 
\begin{equation*}
\nu _{\ell }\left\{ r_{\ell }\colon \left\vert r_{\ell }^{2}-1\right\vert
\geq \lambda \right\} =\Pr \left( \left\vert R_{\ell }^{2}-\mathbb{E}\left[
R_{\ell }^{2}\right] \right\vert \geq \lambda \right) \leq \frac{2}{n_{\ell
d}\lambda ^{2}}.
\end{equation*}%
We show how to bound the other two measures for the case where $u\geq 0$
(the negative case follows analogously). As $\left\vert r_{\ell
}^{2}-1\right\vert <\lambda $ is equivalent to $\sqrt{1-\lambda }<r_{\ell }<%
\sqrt{1+\lambda }$ (recall that $\lambda \in (0,1)$), the monotonicity and
invariance under scaling of $g_{\ell }$ yields 
\begin{align}
\mu _{\ell }^{\ast }& \left\{ (r_{\ell },\alpha _{\ell })\colon g_{\ell
}(h_{\alpha ,\ell },u)-\Psi _{\ell }(u)\geq \varepsilon ,\,\left\vert
r_{\ell }^{2}-1\right\vert <\lambda \right\}   \notag  \label{eq:6} \\
& =\mu _{\ell }^{\ast }\left\{ (r_{\ell },\alpha _{\ell })\colon g_{\ell
}(r_{\ell }\,h_{\alpha ,\ell },r_{\ell }\,u)-\Psi _{\ell }(u)\geq
\varepsilon ,\,\left\vert r_{\ell }^{2}-1\right\vert <\lambda \right\}  
\notag \\
& \leq \mu _{\ell }^{\ast }\left\{ (r_{\ell },\alpha _{\ell })\colon g_{\ell
}(r_{\ell }\,h_{\alpha ,\ell },\sqrt{1-\lambda }\,u)-\Psi _{\ell }(u)\geq
\varepsilon ,\,\left\vert r_{\ell }^{2}-1\right\vert <\lambda \right\}  
\notag \\
& \leq \mu _{\ell }^{\ast }\left\{ (r_{\ell },\alpha _{\ell })\colon g_{\ell
}(r_{\ell }\,h_{\alpha ,\ell },\sqrt{1-\lambda }\,u)-\Psi _{\ell }(u)\geq
\varepsilon \right\} 
\end{align}

Now note that, using the regularity property \ref{eq:16} and the Mean Value
Theorem 
\begin{equation*}
\left\vert \Psi _{\ell }(u)-\Psi _{\ell }(\sqrt{1-\lambda }u)\right\vert
\leq c\left( 1-\sqrt{1-\lambda }\right) \leq c\lambda ,
\end{equation*}%
where $c>0$ does not depend on $u$. Therefore, we can continue to bound~%
\eqref{eq:6} by 
\begin{align*}
& \leq \mu _{\ell }^{\ast }\left\{ (r_{\ell },\alpha _{\ell })\colon g_{\ell
}(r_{\ell }\,h_{\alpha ,\ell },\sqrt{1-\lambda }\,u)-\Psi _{\ell }(\sqrt{%
1-\lambda }\,u)\geq \varepsilon -c\lambda \right\} \\
& =\Pr \left( g_{\ell }(T_{\ell },\sqrt{1-\lambda }\,u)-\Psi _{\ell }(\sqrt{%
1-\lambda }\,u)\geq \varepsilon -c\lambda \right) \\
& \leq \frac{\sigma _{\ell }^{2}}{\left( \varepsilon -c\lambda \right) ^{2}},
\end{align*}%
where we have applied Chebyshev's inequality to arrive at the last bound.
Similarly, we deduce that 
\begin{align*}
\mu _{\ell }^{\ast }& \left\{ (r_{\ell },\alpha _{\ell })\colon \Psi _{\ell
}(u)-g_{\ell }(h_{\alpha ,\ell },u)\leq -\varepsilon ,\,\left\vert r_{\ell
}^{2}-1\right\vert <\lambda \right\} \\
& \leq \Pr \left( g_{\ell }(T_{\ell },\sqrt{1+\lambda }\,u)-\Psi _{\ell }(%
\sqrt{1+\lambda }\,u)\geq \varepsilon -c\lambda \right) \\
& \leq \frac{\sigma _{\ell }^{2}}{\left( \varepsilon -c\lambda \right) ^{2}}.
\end{align*}%
Plugged back into~\eqref{eq:18}, we obtain 
\begin{equation*}
\mu _{\ell }(G_{\ell }(\varepsilon ,u))\leq 2\left( \frac{1}{n_{\ell
d}\lambda ^{2}}+\frac{\sigma _{\ell }^{2}}{(\varepsilon -c\lambda )^{2}}%
\right) .
\end{equation*}%
Choosing $\lambda =\varepsilon /(c+1)$ and taking the supremum completes the
proof.
\end{proof}

Our argument to follow establishes an upper bound on the $L^{\infty }$-norms
of Gaussian eigenfunctions by means of metric entropy ideas and the
Borel-TIS inequality. The computations require a careful analysis of the
high-frequency behaviour of Gegenbauer polynomials. For the lower bound, we
construct a nearly equi-spaced grid of points, where values of the random
field can be viewed as asymptotically independent, and then use standard
results on the supremum of i.i.d. Gaussian variables.

\begin{proof}[Proof of Proposition~\protect\ref{angela}]
Let us define the canonical metric (see i.e. \cite[p.12]{adlertaylor}) on $%
S^{d}$ by 
\begin{equation*}
d_{\ell }(x,y)=\sqrt{\mathbb{E}\left[ T_{\ell }(x)-T_{\ell }(y)\right] ^{2}},
\end{equation*}%
which, by isotropy of $T_{\ell }$, can also be written as 
\begin{equation*}
d_{\ell }(x,y)=\sqrt{2-2G_{\ell ;d}(\cos \vartheta ))},
\end{equation*}%
where $G_{\ell ;d}$ is the $\ell $th Gegenbauer polynomial (see i.e. \cite%
{szego}) and $\vartheta =\func{arccos}\left\langle x,y\right\rangle _{%
\mathbb{R}^{d}}$ is the usual geodesic distance on $S^{d}$.

By Hilb's asymptotic formula for Jacobi polynomials (see for example~\cite[%
Thm. 8.21.12]{szego}), we have uniformly for $\ell \geq 1$, $\vartheta \in
\lbrack 0,\tfrac{\pi }{2}]$ that 
\begin{multline}
G_{\ell ;d}(\cos \vartheta ) \\
=\frac{2^{\frac{d}{2}-1}}{{\binom{\ell +\frac{d}{2}-1}{l}}}(\sin \vartheta
)^{-\frac{d}{2}+1}\left( a_{{\ell },d}\left( \frac{\vartheta }{\sin
\vartheta }\right) ^{\tfrac{1}{2}}J_{\frac{d}{2}-1}(L\vartheta )+\delta
(\vartheta )\right) \ ,  \label{eq:14}
\end{multline}%
where $L=\ell +\frac{d-1}{2}$, $J_{\alpha }$ denotes the Bessel function of
the first kind of order $\alpha $, 
\begin{equation*}
a_{\ell ,d}=\frac{\Gamma (\ell +\frac{d}{2})}{(\ell +\frac{d-1}{2})^{\tfrac{d%
}{2}-1}\ell !}\ \sim \ 1
\end{equation*}%
as $\ell \to \infty$ and the remainder is given by 
\begin{equation*}
\delta (\vartheta )=%
\begin{cases}
O\left( \sqrt{\vartheta }\,{\ell }^{-\tfrac{3}{2}}\ \right) & \qquad (K\ell
)^{-1}<\vartheta <\tfrac{\pi }{2}\ , \\ 
O\left( \vartheta ^{\left( \tfrac{d}{2}-1\right) +2}\,{\ell }^{\tfrac{d}{2}%
-1}\ \right) & \qquad 0<\vartheta <(K\ell )^{-1}\ ,%
\end{cases}%
\end{equation*}%
for some $K>0$.

Therefore, using the asymptotic relation 
\begin{equation*}
\frac{2^{\frac{d}{2}-1}}{{\binom{\ell +\frac{d}{2}-1}{\ell }}}=\frac{2^{%
\frac{d}{2}-1}\left( \frac{d}{2}-1\right) !}{l^{\frac{d}{2}-1}}+o(1),
\end{equation*}%
we have for $\vartheta <(K\ell )^{-1}$ that 
\begin{multline*}
G_{\ell ;d}(\cos \vartheta ) \\
=\frac{2^{\frac{d}{2}-1}\left( \frac{d}{2}-1\right) !}{{\ell }^{\frac{d}{2}%
-1}}(\sin \vartheta )^{-\frac{d}{2}+1}\left( a_{\ell ,d}\left( \frac{%
\vartheta }{\sin \vartheta }\right) ^{\tfrac{1}{2}}J_{\frac{d}{2}%
-1}(L\vartheta )+\delta (\vartheta )\right)
\end{multline*}%
with 
\begin{equation*}
\delta (\vartheta )=O\left( \vartheta ^{2}\right) \ .
\end{equation*}%
Recall that $J_{\alpha }$ is defined as 
\begin{equation*}
J_{\alpha }(x) =\sum_{m=0}^{+\infty }\frac{(-1)^{m}}{m!\Gamma (m+\alpha +1)}%
\left( \frac{x}{2}\right) ^{2m+\alpha },
\end{equation*}%
which implies that 
\begin{equation}
J_{\frac{d}{2}-1}(x)=\left( \frac{x}{2}\right) ^{\frac{d}{2}%
-1}\sum_{m=0}^{+\infty }\frac{(-1)^{m}}{m!\Gamma \left( m+\frac{d}{2}\right) 
}\left( \frac{x}{2}\right) ^{2m}\ .  \label{eqb}
\end{equation}%
Hence, 
\begin{align*}
J_{\frac{d}{2}-1}(x)& =\frac{x^{\frac{d}{2}-1}}{2^{\frac{d}{2}-1}\left( 
\frac{d}{2}-1\right) !}-\frac{x^{\frac{d}{2}+1}}{2^{\frac{d}{2}+1}\left( 
\frac{d}{2}\right) !} \\
&\qquad \qquad+\left( \frac{x}{2}\right) ^{\frac{d}{2}-1}\sum_{m=2}^{+\infty
}\frac{(-1)^{m}}{m!\Gamma \left( m+\frac{d}{2}\right) }\left( \frac{x}{2}%
\right) ^{2m} \\
& =\frac{x^{\frac{d}{2}-1}}{2^{\frac{d}{2}-1}\left( \frac{d}{2}-1\right) !}
\left( 1-\frac{x^{2}}{2d}\right) \\
&\qquad \quad+\left( \frac{x}{2}\right) ^{\frac{d}{2}-1}\sum_{m=2}^{+\infty }%
\frac{(-1)^{m}}{m!\Gamma \left( m+\frac{d}{2}\right) }\left( \frac{x}{2}%
\right) ^{2m}.
\end{align*}

Note that 
\begin{equation*}
\lim_{K\rightarrow +\infty }\sup_{x\leq (K)^{-1}}\left\vert \frac{1-\frac{2^{%
\frac{d}{2}-1}\left( \frac{d}{2}-1\right) !}{x^{\frac{d}{2}-1}}J_{\frac{d}{2}%
-1}(x)}{x^{2}}-\frac{1}{2d}\right\vert =0.
\end{equation*}%
This implies that for $\varepsilon >0$ there exists $K_{\varepsilon }>0$
such that 
\begin{equation*}
\left( \frac{1}{2d}-\varepsilon \right) {\ell }^{2}\vartheta ^{2}\leq 1-%
\frac{2^{\frac{d}{2}-1}\left( \frac{d}{2}-1\right) !}{(L\vartheta )^{\frac{d%
}{2}-1}}J_{\frac{d}{2}-1}(L\vartheta )\leq \left( \frac{1}{2d}+\varepsilon
\right) {\ell }^{2}\vartheta ^{2}
\end{equation*}%
for $\vartheta <K_{\varepsilon }{\ell }^{-1}$. Plugged back into~%
\eqref{eq:14}, we get for $\vartheta <K_{\epsilon }{\ell }^{-1}$ that 
\begin{equation*}
\displaylines{ \left (\frac{1}{2d} -\varepsilon \right ){\ell}^2 \vartheta^2
+ O(\vartheta^2)\le 1-G_{\ell;d}(\cos \vartheta) \le \left (\frac{1}{2d}
+\varepsilon \right ){\ell}^2 \vartheta^2+O(\vartheta^2)\ . }
\end{equation*}%
Therefore, there exist two constants $c_{1},c_{2}>0$ such that for any $%
x,y\in S^{d}$ it holds that 
\begin{equation*}
c_{1}d^{2}(x,y)\leq \frac{d_{{\ell }}^{2}(x,y)}{{\ell }^{2}}\leq
c_{2}d^{2}(x,y).
\end{equation*}%
From here we use exactly the same spherical cap argument as in the proof of
Proposition 2 in \cite{MaVa2}. To briefly recall, we consider some sequence
of balls of radius $\varepsilon $ in the canonical metric, which indeed are
hyperspherical caps whose radius is asymptotically equal to $\frac{%
\varepsilon }{\ell }$, $\varepsilon >0$. Their Euclidean volume is
asymptotically equal to $\frac{\varepsilon ^{d}}{\ell ^{d}}$, therefore the
number $N_{\ell }(\varepsilon )$ of such caps needed to cover the
hypersphere is asymptotically equal to $\frac{\ell ^{d}}{\varepsilon ^{d}}$.
By~\cite[Thm. 1.3.3]{adlertaylor}, there exists a constant $K^{\ast }$, only
depending on $d$, such that 
\begin{equation*}
\displaylines{ \mathbb{E}\left [\sup_{ S^d} T_\ell \right ] \le K^* \Big (
\int_{0}^{C/\ell} \sqrt{\log N_\ell(\varepsilon)}\, \mathrm{d}\varepsilon
+\int_{C/\ell}^{\delta} \sqrt{\log N_\ell(\varepsilon)}\,
\mathrm{d}\varepsilon \Big )\ . }
\end{equation*}%
Note that analogous steps yield the upper bound $c\sqrt{2d\log \ell }$ for
both of the previous summands, where $c>0$ is some constant.

To prove part (ii), we apply the Borel-TIS inequality (see for example Thm.
2.1.1, p.50 in \cite{adlertaylor}), which yields for $t>\func{E}\left[
\left\Vert T_{\ell }\right\Vert _{\infty }\right] $ that 
\begin{equation}
\Pr \left( \left\Vert T_{\ell }\right\Vert _{\infty }>t\right) \leq \mathrm{e%
}^{-\left( t-\mathbb{E}\left[ \left\Vert T_{\ell }\right\Vert _{\infty }%
\right] \right) ^{2}/2}.  \label{eq:2}
\end{equation}%
The result now follows from part (i). The proof of part (iii) uses a similar
argument as given in the case of needlet random fields by \cite{MaVa2}. In
particular, we have that%
\begin{equation}
\Pr \left( \sup_{x\in S^{d}}\left\vert T_{\ell }(x)\right\vert \geq K\sqrt{%
\log \ell }\right) \geq \Pr \left( \sup_{\xi _{k}\in \Xi _{\ell }}\left\vert
T_{\ell }(\xi _{k})\right\vert \geq K\sqrt{\log \ell }\right)   \label{eq:22}
\end{equation}%
where, for some $\alpha \in (0,1)$, $\Xi _{\ell }$ is a grid of points $%
\left\{ \xi _{k}\right\} $ such that 
\begin{equation*}
\min_{k\neq k^{\prime }}d(\xi _{k},\xi _{k^{\prime }})\geq \ell ^{-\alpha }
\end{equation*}%
and such that $\#\Xi _{\ell }=N_{\ell }$ is of order $\ell ^{d\alpha }$. The
existence of such grids is well-known and has for instance been exploited in
the construction of cubature points for spherical wavelets (see \cite%
{NPW1,BKMPBer}).

Now define events $A$ and $B$ by 
\begin{align*}
A& =\left\{ \left\Vert \left\{ Z_{\ell 1},...Z_{\ell ,N_{\ell }}\right\}
\right\Vert _{\infty }\geq 2K\sqrt{\log \ell }\right\} \\
B& =\left\{ \left\Vert \left\{ Z_{\ell 1},...Z_{\ell ,N_{\ell }}\right\}
-\left\{ T_{\ell }(\xi _{1}),...T_{\ell }(\xi _{N_{\ell }})\right\}
\right\Vert _{\infty }\leq K\sqrt{\log \ell }\right\}
\end{align*}%
Then $A\cap B$ implies the event 
\begin{equation*}
\left\{ \sup_{\xi _{k}\in \Xi _{\ell }}\left\vert T_{\ell }(\xi
_{k})\right\vert \geq K\sqrt{\log \ell }\right\} .
\end{equation*}%
Therefore, 
\begin{align}
\Pr \left( \sup_{\xi _{k}\in \Xi _{\ell }}\left\vert T_{\ell }(\xi
_{k})\right\vert \geq K\sqrt{\log \ell }\right) & \geq \Pr \left( A\cap
B\right)  \notag \\
& \geq \Pr \left( A\right) +\Pr \left( B\right) -1  \notag \\
& =1-\Pr \left( A^{c}\right) -\Pr \left( B^{c}\right) ,  \label{eq:20}
\end{align}%
which, together with~\eqref{eq:22}, implies that 
\begin{equation}
\Pr \left( \left\Vert T_{\ell }\right\Vert \leq K\sqrt{\log \ell }\right)
\leq \Pr (A^{c})+\Pr (B^{c}).  \label{eq:23}
\end{equation}

Using Mill's inequality 
\begin{equation*}
\frac{2z}{1+z^{2}}\phi (z)\leq \Pr \left\{ Z>z\right\} \leq \frac{2}{z}\phi
(z),
\end{equation*}%
we can evaluate the probability $\Pr (A^{c})$ as follows. 
\begin{align*}
\Pr (A^{c})& =\Pr \left( \left\Vert \left\{ Z_{\ell 1},...Z_{\ell ,N_{\ell
}}\right\} \right\Vert _{\infty }<2K\sqrt{\log \ell }\right) \\
& =\prod\limits_{k=1}^{N_{\ell }}\Phi (2K\sqrt{\log \ell })\simeq \left\{ 1-%
\frac{1}{2K\sqrt{\log \ell }}\phi (2K\sqrt{\log \ell })\right\} ^{N_{\ell }}
\\
& =\left\{ 1-\frac{1}{2K\sqrt{\log \ell }}\frac{1}{\sqrt{2\pi }}\exp (-\frac{%
1}{2}(2K\sqrt{\log \ell })^{2})\right\} ^{N_{\ell }} \\
& =\left\{ 1-\frac{1}{\sqrt{8\pi }K^{2}\ell ^{2K^{2}}\sqrt{\log \ell }}%
\right\} ^{\ell ^{d\alpha }}.
\end{align*}%
Using the asymptotics $\log (1+z)\sim z$ for small $z$, we get that 
\begin{equation}
\Pr (A^{c})\sim \exp \left( -\frac{\ell ^{d\alpha -2K^{2}}}{\sqrt{8\pi }K^{2}%
\sqrt{\log \ell }}\right) .  \label{eq:21}
\end{equation}

To estimate the probability $\func{P}(B^{c})$, write $\Sigma _{\ell }$ for
the covariance matrix of the vector $\left\{ T_{\ell }(\xi _{1}),...T_{\ell
}(\xi _{N_{\ell }})\right\} ,$ and $I_{N_{\ell }}$ for the identity matrix
of order $N_{\ell };$ it should be recalled that, due to our choice of the
grid $\Xi _{\ell }$, the elements of the vector $Z_{\ell }$ are
asymptotically independent. Note also that%
\begin{align*}
\Sigma _{\ell }^{-1/2}-I_{N_{\ell }}& =Q\Lambda ^{-1/2}Q^{\ast }-I_{N_{\ell
}} \\
& =Q\left( \Lambda ^{-1/2}-I_{N_{\ell }}\right) Q^{\ast } \\
& =Q\left( \left( \Lambda ^{-1/2}-I_{N_{\ell }}\right) \left( \Lambda
^{-1/2}+I_{N_{\ell }}\right) \left( \Lambda ^{-1/2}+I_{N_{\ell }}\right)
^{-1}\right) Q^{\ast } \\
& =Q\left( \Big(\Lambda ^{-1}-I_{N_{\ell }}\Big)\left( \Lambda
^{-1/2}+I_{N_{\ell }}\right) ^{-1}\right) Q^{\ast } \\
& =Q\left( \Lambda ^{-1}\Big(I_{N_{\ell }}-\Lambda \Big)\left( \Lambda
^{-1/2}+I_{N_{\ell }}\right) ^{-1}\right) Q^{\ast },
\end{align*}%
and thus 
\begin{eqnarray*}
\left\Vert \Sigma _{\ell }^{-1/2}-I_{N_{\ell }}\right\Vert _{2}^{2} &\leq
&\left\Vert Q\left( \Lambda ^{-1}\Big(I_{N_{\ell }}-\Lambda \Big)\left(
\Lambda ^{-1/2}+I_{N_{\ell }}\right) ^{-1}\right) Q^{\ast }\right\Vert
_{2}^{2} \\
&\leq &\left\Vert Q\right\Vert _{2}^{2}\left\Vert Q^{\ast }\right\Vert
_{2}^{2}\left\Vert \Lambda ^{-1}\right\Vert _{2}^{2}\left\Vert \left(
\Lambda ^{-1/2}+I_{N_{\ell }}\right) ^{-1}\right\Vert _{2}^{2}\left\Vert
I_{N_{\ell }}-\Lambda \right\Vert _{2}^{2} \\
&=&O(\left\Vert I_{N_{\ell }}-\Lambda \right\Vert _{2}^{2}).=O\left( \frac{%
N_{\ell }^{2}}{\sqrt{\ell d(\xi _{k},\xi _{k^{\prime }})}}\right) \\
&=&O\left( \frac{N_{\ell }^{2}}{\ell ^{(1-\alpha )/2}}\right) =O(\ell
^{2d\alpha +\alpha /2-1/2}).
\end{eqnarray*}%
Consequently, for the i.i.d. standard Gaussian array%
\begin{equation*}
\left( Z_{\ell 1},...Z_{\ell ,N_{\ell }}\right) =\Sigma _{\ell
}^{-1/2}\left( T_{\ell }(\xi _{1}),...T_{\ell }(\xi _{N_{\ell }})\right)
\end{equation*}%
it holds that 
\begin{equation*}
\left\Vert \left( Z_{\ell 1},...Z_{\ell ,N_{\ell }}\right) -\left( T_{\ell
}(\xi _{1}),...T_{\ell }(\xi _{N_{\ell }})\right) \right\Vert _{\infty
}^{2}\leq \left\Vert \left( Z_{\ell 1},...Z_{\ell ,N_{\ell }}\right) -\left(
T_{\ell }(\xi _{1}),...T_{\ell }(\xi _{N_{\ell }})\right) \right\Vert
_{2}^{2},
\end{equation*}%
and for the expectation of the right hand side we have 
\begin{eqnarray*}
&&\mathbb{E}\left\Vert \left( Z_{\ell 1},...Z_{\ell ,N_{\ell }}\right)
-\left( T_{\ell }(\xi _{1}),...T_{\ell }(\xi _{N_{\ell }})\right)
\right\Vert _{2}^{2} \\
&\leq &\left\Vert \Sigma _{\ell }^{-1/2}-I_{N_{\ell }}\right\Vert _{2}^{2}%
\mathbb{E}\left\Vert \left( T_{\ell }(\xi _{1}),...T_{\ell }(\xi _{N_{\ell
}})\right) \right\Vert _{2}^{2} \\
&\leq &\left\Vert \Sigma _{\ell }^{-1/2}-I_{N_{\ell }}\right\Vert
_{2}^{2}N_{\ell } \\
&\leq &O(N_{\ell }\times \ell ^{2d\alpha +\alpha /2-1/2}) \\
&=&O(\ell ^{3d\alpha +\alpha /2-1})=O(\ell ^{((6d+1)\alpha -1)/2})\text{ .}
\end{eqnarray*}%
By Chebyshev's inequality, we thus get 
\begin{align}
\Pr (B^{c})& =\Pr \left( \left\Vert \left( Z_{\ell 1},...Z_{\ell ,N_{\ell
}}\right) -\left( T_{\ell }(\xi _{1}),...T_{\ell }(\xi _{N_{\ell }})\right)
\right\Vert _{2}\geq K\sqrt{\log \ell }\right)  \notag \\
& \leq \frac{1}{K\log \ell }E\left\Vert \left( Z_{\ell 1},...Z_{\ell
,N_{\ell }}\right) -\left( T_{\ell }(\xi _{1}),...T_{\ell }(\xi _{N_{\ell
}})\right) \right\Vert _{2}^{2}  \notag \\
& =O\left( \frac{\ell ^{((6d+1)\alpha -1)/2}}{\log \ell }\right) .
\label{eq:25}
\end{align}%
Plugging~\eqref{eq:21} and~\eqref{eq:25} back into~\eqref{eq:23} yields for
any $\alpha \in (0,1)$ that 
\begin{equation*}
\Pr \left( \left\Vert T_{\ell }\right\Vert \leq K\sqrt{\log \ell }\right)
=O\left( \exp \left( -\frac{\ell ^{d\alpha -2K^{2}}}{\sqrt{8\pi }K^{2}\sqrt{%
\log \ell }}\right) +\frac{\ell ^{((6d+1)\alpha -1)/2}}{\log \ell }\right) .
\end{equation*}%
Now note that the right hand side tends to zero if and only if we have $0<K<%
\sqrt{d/(12d+2)}$ and $2K^{2}/d<\alpha <1/(6d+1)$. Furthermore, in this case
the exponential term on the right hand side is dominated and can be
neglected.
\end{proof}

The idea in the proof below is again to associate a Gaussian measure to
Lebesgue by introducing a random radius. In this case, however, we shall
exploit a Large Deviation Principle on the radius itself, and because of
this we will obtain a sharper bound on the rate of convergence to zero for
the involved measures.

\begin{proof}[Proof of Theorem~\protect\ref{inftynorm}]
We have for all $K,a>0$ that 
\begin{align*}
\mu _{\ell } &\left\{ \alpha _{\ell } \colon \left\Vert h_{\alpha ,\ell
}\right\Vert_{\infty}\geq K \sqrt{\log \ell} \right\} \\
&= \mu_{\ell }^{\ast }\left\{ (r_{\ell},\alpha _{\ell }) \colon \left\Vert
h_{\alpha ,\ell }\right\Vert_{\infty}\geq K \sqrt{\log \ell} \right\} \\
&= \mu_{\ell }^{\ast }\left\{ (r_{\ell},\alpha _{\ell }) \colon \left\Vert
r_{\ell} \, h_{\alpha ,\ell }\right\Vert_{\infty}\geq r_{\ell} K \sqrt{\log
\ell} \right\} \\
&= \mu_{\ell }^{\ast }\left\{ (r_{\ell},\alpha _{\ell }) \colon \left\Vert
r_{\ell} \, h_{\alpha ,\ell }\right\Vert_{\infty}\geq r_{\ell} K \sqrt{\log
\ell}, \, r_{\ell} \geq a^{-1} \right\} \\
& \qquad + \mu_{\ell }^{\ast }\left\{ (r_{\ell},\alpha _{\ell }) \colon
\left\Vert r_{\ell} \, h_{\alpha ,\ell }\right\Vert_{\infty}\geq r_{\ell} K 
\sqrt{\log \ell}, \, r_{\ell} < a^{-1} \right\} \\
&\leq \mu_{\ell }^{\ast }\left\{ (r_{\ell},\alpha _{\ell }) \colon
\left\Vert r_{\ell} \, h_{\alpha ,\ell }\right\Vert_{\infty}\geq \frac{K}{a} 
\sqrt{\log \ell} \right\} + \nu_{\ell} \left\{ r_{\ell} < a^{-1} \right\}.
\end{align*}%
Now set $K= M + \sqrt{2\beta} $ and $a = \left( M+ \sqrt{2\beta} \right)/
\left( M + \sqrt{2\beta^{\prime }} \right) > 1$. Then, by Lemma~\ref{LDP},
the second measure in the above bound vanishes exponentially in $\ell$. The
result now follows from part (ii) of Theorem~\ref{angela}. To prove part
(ii), we proceed analogously and write 
\begin{align*}
\mu _{\ell } &\left\{ \alpha _{\ell } \colon \left\Vert h_{\alpha ,\ell
}\right\Vert_{\infty}\leq K^{\prime }\sqrt{\log \ell} \right\} \\
&= \mu _{\ell }^{\ast} \left\{ (r_{\ell},\alpha_{\ell }) \colon \left\Vert
h_{\alpha ,\ell}\right\Vert_{\infty}\leq K^{\prime }\sqrt{\log \ell} \right\}
\\
&= \mu _{\ell }^{\ast} \left\{ (r_{\ell},\alpha_{\ell }) \colon \left\Vert
r_{\ell} \, h_{\alpha ,\ell}\right\Vert_{\infty}\leq r_{\ell} K^{\prime }%
\sqrt{\log \ell} \right\} \\
&= \mu _{\ell }^{\ast} \left\{ (r_{\ell},\alpha_{\ell }) \colon \left\Vert
r_{\ell} \, h_{\alpha ,\ell}\right\Vert_{\infty}\leq r_{\ell} K^{\prime }%
\sqrt{\log \ell}, \, r_{\ell} \leq a^{-1} \right\} \\
& \qquad + \mu _{\ell }^{\ast} \left\{ (r_{\ell},\alpha_{\ell }) \colon
\left\Vert r_{\ell} \, h_{\alpha ,\ell}\right\Vert_{\infty}\leq r_{\ell}
K^{\prime }\sqrt{\log \ell}, \, r_{\ell} > a^{-1} \right\} \\
&\leq \mu _{\ell }^{\ast} \left\{ (r_{\ell},\alpha_{\ell }) \colon
\left\Vert r_{\ell} \, h_{\alpha ,\ell}\right\Vert_{\infty}\leq \frac{%
K^{\prime }}{a} \sqrt{\log \ell} \right\} + \nu_{\ell} \left\{ r_{\ell} >
a^{-1} \right\}.
\end{align*}
For any $a>1$, the second measure vanishes exponentially in $\ell$ by Lemma~%
\ref{LDP} and the result now follows from part (iii) of Theorem~\ref{angela}
(clearly, given $K^{\prime }\in (0,\sqrt{d/(12d+2)})$, we can find $K$ in
the same interval and $a>1$ such that $K^{\prime }=K/a$).

\end{proof}

\begin{proof}[Proof of Corollary~\protect\ref{cor:3}]
In the setting of Theorem~\ref{inftynorm}, choose $\beta ^{\prime }$ such
that $1<\beta ^{\prime }<\beta $. Then, the Borel-Cantelli Lemma gives 
\begin{align*}
\mu _{\infty }& \left\{ (\alpha _{\ell })_{\ell \geq 1}\colon \left\Vert
h_{\alpha ,\ell }\right\Vert _{\infty }\geq \left( M+\sqrt{2\beta }\right) 
\sqrt{\log \ell }\text{ infinitely often}\right\} \\
& \leq \lim_{L\rightarrow \infty }\sum_{\ell =L}^{\infty }\mu _{\ell
}\left\{ \alpha _{\ell }\colon \left\Vert h_{\alpha ,\ell }\right\Vert
_{\infty }\geq (M+\sqrt{2\beta })\sqrt{\log \ell }\right\} \\
& \leq \text{const}\times \lim_{L\rightarrow \infty }\sum_{\ell =L}^{\infty }%
\frac{1}{\ell ^{\beta ^{\prime }}}=0\text{ .}
\end{align*}
Identity~\eqref{eq:10} can be show in the same way.
\end{proof}

Finally, we provide the proof for the convergence in Kolmogorov distance.
This requires some uniform bound, which is obtained by means of chaining
arguments, similar for instance to the one given by \cite{dehlingtaqqu}.
Details require some care, because in this setting we are also covering the
case where the dimension $d$ grows to infinity.

\begin{proof}[Proof of Theorem~\protect\ref{prop:1}]
We note first that%
\begin{align*}
d_{Kol}&(h_{\alpha ,\ell }(X),Z) \\
&=\sup_{u}\left\vert \int_{S^{d}}1_{(-\infty ,u]}\left( h_{\alpha ,\ell
}(x)\right) \mathrm{d}x-\Phi (u)\right\vert \\
& =\sup_{u}\left\vert \int_{S^{d}}\left( 1_{(-\infty ,u]}\left( h_{\alpha
,\ell }(x)\right) -1_{(-\infty ,x]}\left( T_{\ell }(x)\right) \right) 
\mathrm{d}x \right. \\
& \qquad \qquad \qquad \qquad\left. +\int_{S^{d}}1_{(-\infty ,u]}\left(
T_{\ell }(x)\right) \mathrm{d}x-\Phi (u)\right\vert \\
& \leq \sup_{u}\left\vert \int_{S^{d}}\left( 1_{(-\infty ,u]}\left(
h_{\alpha ,\ell }(x)\right) -1_{(-\infty ,u]}\left( T_{\ell }(x)\right)
\right) \mathrm{d}x\right\vert \\
& \qquad \qquad \qquad \qquad+\sup_{u}\left\vert \int_{S^{d}}1_{(-\infty
,u]}\left( T_{\ell }(x)\right) \mathrm{d}x-\Phi (u)\right\vert ,
\end{align*}%
where $\Phi $ denotes the standard Gaussian cumulative distribution
function. Now%
\begin{align*}
\mu _{\ell }& \left\{ \sup_{u}\left\vert \int_{S^{d}}\left( 1_{(-\infty
,u]}\left( h_{\alpha ,\ell }(x)\right) -\Phi (u)\right) \mathrm{d}%
x\right\vert >\varepsilon _{\ell }\right\} \\
& =\mu _{\ell }^{\ast }\left\{ \sup_{u}\left\vert \int_{S^{d}}\left(
1_{(-\infty ,ru]}\left( r_{\ell}h_{\alpha ,\ell }(x)\right) -\Phi (u)\right) 
\mathrm{d}x\right\vert >\varepsilon _{\ell }\right\} \\
& =\mu _{\ell }^{\ast }\left\{ \sup_{u}\left\vert \int_{S^{d}}\left(
1_{(-\infty ,ru]}\left( T_{\ell }(x)\right) -\Phi (u)\right) \mathrm{d}%
x\right\vert >\varepsilon _{\ell }\right\} \\
& =\mu _{\ell }^{\ast }\left\{ \sup_{u}\left\vert \int_{S^{d}}\left(
1_{(-\infty ,ru]}\left( T_{\ell }(x)\right) -\Phi (ur_{\ell})+\Phi
(ur_{\ell})-\Phi (u)\right) \mathrm{d}x\right\vert >\varepsilon _{\ell
}\right\} \\
& \leq \mu _{l}^{\ast }\left\{ \sup_{u}\left\vert \int_{S^{d}}\left(
1_{(-\infty ,u]}\left( T_{\ell }(x)\right) -\Phi (u)\right) \mathrm{d}%
x\right\vert >\frac{\varepsilon _{\ell }}{2}\right\} \\
& \qquad \qquad +\mu _{\ell }^{\ast }\left\{ \sup_{u}\left\vert \Phi
(ur_{\ell})-\Phi (u)\right\vert >\frac{\varepsilon _{\ell }}{2}\right\} ,
\end{align*}%
where we used the inclusion%
\begin{equation*}
\left\{ (x,y)\in \mathbb{R}^{2}\colon \left\vert x-y\right\vert >\varepsilon
\right\} \subseteq \left\{ (x,y)\in \mathbb{R}^{2}\colon \left\vert
x-z\right\vert >\frac{\varepsilon }{2}\right\} \cup \left\{ (x,y)\in \mathbb{%
R}^{2}\colon \left\vert y-z\right\vert >\frac{\varepsilon }{2}\right\} ,
\end{equation*}%
which is valid for any $z\in \mathbb{R}$.

By concavity of $\Phi $ and a Taylor argument, we obtain for the second term 
\begin{equation*}
\sup_{u}\left\vert \Phi (ur_{\ell})-\Phi (u)\right\vert \leq
\sup_{u}\left\vert \phi (u)u\right\vert \left\vert r_{\ell }-1\right\vert .
\end{equation*}%
Hence it holds that 
\begin{align*}
\mu _{\ell }^{\ast }\left\{ \sup_{u}\left\vert \Phi (ur_{\ell })-\Phi
(u)\right\vert > \frac{\varepsilon_{\ell}}{2}\right\} & =\nu _{\ell }\left\{
\sup_{u}\left\vert \Phi (ur_{\lambda})-\Phi (u)\right\vert > \frac{%
\varepsilon_{\ell}}{2} \right\} \\
& \leq \nu _{\ell }\left\{ \left\vert r_{\ell}-1\right\vert >\varepsilon
_{\ell }\sqrt{\frac{\pi}{2} }\right\} \\
& \leq \frac{2 \func{E}_{\nu _{\ell }}\left[ (R_{\ell }-1)^{2}\right] }{%
\varepsilon _{\ell }^{2}\pi } \\
& \leq \frac{K}{n_{\ell d}\varepsilon _{\ell }^{2}}
\end{align*}

For the first term, we follow an argument similar to the standard proof of
uniform convergence for Glivenko-Cantelli theorems (see for instance \cite%
{dehlingtaqqu}). In particular, let us choose an array of refining
partitions $\left\{ u_{k;\ell }\right\} $ such that $\Phi (u_{k;\ell })-\Phi
(u_{k-1;\ell })<\varepsilon _{\ell }/6$. Then, given $u$, we can find $k$
such that $u_{k-1}\leq u\leq u_{k}$. As cumulative distribution functions
are increasing, we thus have that 
\begin{align*}
\int_{S^{d}}\left( 1_{(-\infty ,u]}\left( T_{\ell }(x)\right) -\Phi
(u)\right) dx& \leq \int_{S^{d}}\left( 1_{(-\infty ,u_{k}]}\left( T_{\ell
}(x)\right) -\Phi (u_{k-1})\right) dx \\
& \leq \int_{S^{d}}\left( 1_{(-\infty ,u_{k}]}\left( T_{\ell }(x)\right)
-\Phi (u_{k})+\frac{\varepsilon _{\ell }}{6}\right) dx
\end{align*}%
Likewise, we have that 
\begin{multline*}
\int_{S^{d}}\left( 1_{(-\infty ,u]}\left( T_{\ell }(x)\right) -\Phi
(u)\right) dx \\
\geq \int_{S^{d}}\left( 1_{(-\infty ,u_{k-1}]}\left( T_{\ell }(x)\right)
-\Phi (u_{k-1})-\frac{\varepsilon _{\ell }}{6}\right) dx
\end{multline*}%
and therefore it holds that 
\begin{multline*}
\sup_{u}\left\vert \int_{S^{d}}\left( 1_{(-\infty ,u]}\left( T_{\ell
}(x)\right) -\Phi (u)\right) dx\right\vert \\
\leq \max_{k}\left\vert \int_{S^{d}}\left( 1_{(-\infty ,u_{k}]}\left(
T_{\ell }(x)\right) -\Phi (u_{k})\right) dx\right\vert +\frac{\varepsilon
_{\ell }}{6}.
\end{multline*}%
Hence, choosing a sequence of partitions $\left\{ u_{k;\ell }\right\} $, we
get that 
\begin{align*}
\mu _{\ell }^{\ast }& \left\{ \sup_{u}\left\vert \int_{S^{d}}\left(
1_{(-\infty ,u]}\left( T_{\ell }(x)\right) -\Phi (u)\right) dx\right\vert >%
\frac{\varepsilon _{\ell }}{2}\right\} \\
& \leq \mu _{\ell }^{\ast }\left\{ \sum_{k}\left\vert \int_{S^{d}}\left(
1_{(-\infty ,u_{k;\ell }]}\left( T_{\ell }(x)\right) -\Phi (u_{k;\ell
})\right) dx\right\vert >\frac{\varepsilon _{\ell }}{3}\right\} \\
& \leq 9\#\left\{ u_{k;\ell }\right\} \,\frac{\max_{k}\mathbb{E}\left[
\left\vert \int_{S^{d}}\left( 1_{(-\infty ,u_{k;\ell }]}\left( T_{\ell
}(x)\right) -\Phi (u_{k;\ell })\right) dx\right\vert \right] ^{2}}{%
\varepsilon _{\ell }^{2}},
\end{align*}%
where $\#\left\{ u_{\ell ;k}\right\} $ denotes the cardinality of the
partition, which clearly is of order $1/\varepsilon _{\ell }$.

Now we recall from \cite{mrossi} that 
\begin{equation*}
\max_{k}\mathbb{E}\left[ \left\vert \int_{S^{d}}\left( 1_{(-\infty
,u_{k;\ell }]}\left( T_{\ell }(x)\right) -\Phi (u_{k;\ell })\right)
dx\right\vert \right] ^{2}\leq \frac{K}{n_{\ell d}},
\end{equation*}%
where $K\leq \sup_{u}\left( \Phi (u)(1-\Phi (u))\right) =1/4$, which
finishes the proof of the first statement of the Theorem. The proof of the
second statement can be given by exactly the same argument - we need only
establish a uniform bound on the behaviour of the variance of the excursion
volume, for $\ell $ fixed and as $d$ goes to infinity. The proof of this
bound is given in Lemma~\ref{lemmaunob} in the appendix.
\end{proof}

\begin{proof}[Proof of Corollary~\protect\ref{cor:1}]
We consider the case where $\ell $ is fixed and $d$ diverges to infinity;
the proof for the other statement in the Corollary is identical.

For fixed $\ell $, it holds that%
\begin{align*}
& \mathbb{E}\left[ d_{Kol}(h_{\alpha ,\ell }(X),Z)\right] \\
& \leq \mu _{\ell }\left\{ d_{Kol}(h_{\alpha ,\ell }(X),Z)<\kappa _{\ell
})\right\} \times \kappa _{\ell }+\int_{\kappa _{\ell }}^{\infty }\mu _{\ell
}\left\{ h_{\alpha ,\ell }(X),Z)>u\right\} du \\
& \leq \kappa _{\ell }+C\int_{\kappa _{\ell }}^{\infty }\frac{1}{n_{\ell
d}u^{3}}du,
\end{align*}%
where the constant $C$ is uniform in $d$, in view of Lemma~\ref{lem:2} in
the appendix and Chebyshev's inequality. It is then easy to conclude that 
\begin{equation*}
\mathbb{E}\left[ d_{Kol}(h_{\alpha ,\ell }(X),Z)\right] \leq \kappa _{\ell }+%
\frac{C}{n_{\ell d}\kappa _{\ell }^{2}}
\end{equation*}%
and choosing $\kappa _{\ell }=1/n_{\ell d}^{1/3}$ together with the
asymptotics $n_{\ell d}\approx d^{\ell }$ for $d\rightarrow \infty $ and $%
n_{\ell d}\approx \ell ^{d-1}$ (see the introduction) finishes the proof.
\end{proof}

\begin{proof}[Proof of Proposition~\protect\ref{prop:2}]
The same argument as in the proof of Theorem~\ref{inftynorm}.iii works here
as well.
\end{proof}

\section{Appendix: Technical lemmas}

The proofs collected in this final Section are all rather standard, and they
have been included mainly for completeness.

\begin{lemma}
\label{LDP} For all $a>1,$ there exist constants $K,a^{\prime }>0$ such that%
\begin{equation*}
\mu _{\ell }^{\ast }\left\{ R_{\ell }^{2}\geq a\right\} +\mu _{\ell }^{\ast
}\left\{ R_{\ell }^{2}\leq a^{-1}\right\} \leq K\exp \left( -\ell a^{\prime
}\right).
\end{equation*}
\end{lemma}

\begin{proof}
Let $(X_{n})$ be a sequence of i.i.d. random variables with common law $\mu $%
, whose Laplace transform $\hat{\mu}$ is finite in a neighborhood of the
origin. Then it is well known that the sequence $(\overline{X}_{n})$
satisfies the large deviation principle 
\begin{align*}
\limsup_{n\rightarrow +\infty }\frac{1}{n}\log \Pr (\overline{X}_{n}\in A)&
\leq -\inf_{x\in \overline{A}}\Lambda ^{\ast }(x)\  \\
\liminf_{n\rightarrow +\infty }\frac{1}{n}\log {\Pr }(\overline{X}_{n}\in
A)& \geq -\inf_{x\in A^{o}}\Lambda ^{\ast }(x)\ 
\end{align*}%
for any Borel set $A$. Here, $\Lambda ^{\ast }$ is the Cram{\'{e}}r
transform 
\begin{equation*}
\Lambda ^{\ast }(x)=\sup_{\theta \in \mathbb{R}}\left( \theta x-\Lambda
(\theta )\right) \ ,
\end{equation*}%
where $\Lambda $ is the logarithm of the Laplace transform $\hat{\mu}$.
Hence we need the Cram{\'{e}}r transform for chi-square distributions. In
particular, we have that $X_{1}\sim \Gamma (\frac{1}{2},\frac{1}{2})$, so
that its Laplace transform is for $\theta <\frac{1}{2}$ 
\begin{equation*}
\widehat{\mu }(\theta )=\frac{\left( \frac{1}{2}\right) ^{\frac{1}{2}}}{%
\left( \frac{1}{2}-\theta \right) ^{\frac{1}{2}}}
\end{equation*}%
and still for $\theta <\frac{1}{2}$ 
\begin{equation*}
\Lambda (\theta )=\log \widehat{\mu }(\theta )=-\frac{1}{2}\log \left(
1-2\theta \right) \ .
\end{equation*}%
Now let us compute its Cram{\'{e}}r transform 
\begin{equation*}
\Lambda ^{\ast }(x)=\sup_{\theta <\frac{1}{2}}\left( \theta x+\frac{1}{2}%
\log \left( 1-2\theta \right) \right) \ .
\end{equation*}%
Note that if $x\leq 0$, then 
\begin{equation*}
\Lambda ^{\ast }(x)=+\infty \ .
\end{equation*}%
Consider $x>0$, then the function $\theta \mapsto \theta x+\frac{1}{2}\log
\left( 1-2\theta \right) $ has unique maximum at the critical point $\theta
^{\ast }$ 
\begin{equation*}
x-\frac{1}{1-2\theta ^{\ast }}=0\quad \Rightarrow \quad \theta ^{\ast }=-%
\frac{1}{2x}+\frac{1}{2}\ .
\end{equation*}%
Therefore%
\begin{eqnarray*}
\Lambda ^{\ast }(x) &=&(-\frac{1}{2x}+\frac{1}{2})x+\frac{1}{2}\log (\frac{1%
}{x}) \\
&=&\frac{1}{2}(x-1-\log x)
\end{eqnarray*}%
Note that the above function is convex, has a unique minimum (equal to $0$)
at $x=1$ (indeed $1$ is the expected value of a chi-square random variable
with one degree of freedom), for $0<x<1$ is strictly decreasing and for $x>1$
is strictly increasing. Therefore we have (in our notation $R_{\ell }^{2}$
is the empirical mean of $2\ell +1$ i.i.d. chi-square random variables with
one degree of freedom)%
\begin{equation*}
\lim_{\ell \rightarrow \infty }\frac{1}{2\ell +1}\log \Pr \left( R_{\ell
}^{2}\geq a\right) =-\inf_{x\geq a}\Lambda ^{\ast }(x)
\end{equation*}%
where 
\begin{equation*}
\inf_{x\geq a}\Lambda ^{\ast }(x)=\left\{ 
\begin{array}{c}
0\text{ , for }a\leq 1 \\ 
\Lambda ^{\ast }(a)\text{ , for }a>1%
\end{array}%
\right. \ .
\end{equation*}%
It hence follows immediately that for all $\delta >0,$ $a>1,$ $\ell $ large
enough we have%
\begin{equation*}
\Pr \left( R_{\ell }^{2}\geq a\right\} \leq K\exp \left\{ -\ell \left[ \frac{%
1}{2}(a-1-\log a)-\delta \right] \right) .
\end{equation*}%
Likewise%
\begin{equation*}
\lim_{\ell \rightarrow \infty }\frac{1}{2\ell +1}\log \Pr \left( R_{\ell
}^{2}\leq \frac{1}{a}\right) =-\inf_{x\leq a^{-1}}\Lambda ^{\ast
}(x)=\left\{ 
\begin{array}{c}
0\text{ , for }a\leq 1 \\ 
-\Lambda ^{\ast }(a^{-1})\text{ , for }a>1%
\end{array}%
\right. ,
\end{equation*}%
so that for all $\delta >0,$ $a<1,$ $\ell $ large enough we have%
\begin{equation*}
\Pr \left( R_{\ell }^{2}\leq a^{-1}\right) \leq K\exp \left( -\ell \left[ 
\frac{1}{2}(a^{-1}-1+\log a)-\delta \right] \right) .
\end{equation*}%
It then suffices to take $\delta $ such that 
\begin{equation*}
a^{\prime }=\min \left\{ \frac{1}{2}(a^{-1}-1+\log a)-\delta ,\frac{1}{2}%
(a-1-\log a)-\delta \right\} >0\text{ ,}
\end{equation*}%
and the proof is completed.
\end{proof}

\subsection{Verification of excursion functional properties for Example~%
\protect\ref{ex:1}}

For the following lemmas, we adopt the setting and notation of Example~\ref%
{ex:1}.

\begin{lemma}
\bigskip \label{lemmaunob} As $\ell \rightarrow \infty ,$ we have%
\begin{equation*}
\func{Var}\left( \int_{S^{d}}\left( 1_{[u,\infty )}\left( T_{\ell
}(x)\right) -\Phi (u_{\ell })\right) dx\right) =O(\frac{1}{n_{\ell ,d}})%
\text{ ,}
\end{equation*}%
uniformly over $u.$
\end{lemma}

\begin{proof}
The $L^{2}$ expansion for the excursion volume is given by (see \cite%
{dehlingtaqqu,MaWi,mrossi})%
\begin{equation*}
\int_{S^{d}}\left( 1_{[u,\infty )}\left( T_{\ell }(x)\right) -\Phi (u_{\ell
})\right) dx=\int_{S^{d}}\sum_{q=2}^{\infty }\frac{J_{q}(u)}{q!}%
H_{q}(T_{\ell }(x))dx,
\end{equation*}%
where%
\begin{equation*}
J_{q}(u):=\Phi ^{(q)}(u)\text{ .}
\end{equation*}%
Hence we have the uniform bound%
\begin{eqnarray*}
\func{Var}\left( \int_{S^{d}}\sum_{q=2}^{\infty }\frac{J_{q}(u)}{q!}%
H_{q}(T_{\ell }(x))dx\right) &=&\sum_{q=2}^{\infty }\frac{J_{q}^{2}(u)}{q!}%
\int_{S^{d}}\int_{S^{d}}\frac{G_{\ell ;d}(\left\langle x,y\right\rangle )}{%
G_{\ell ;d}(1)}dxdy \\
&\leq &\frac{K}{n_{\ell ,d}}\sum_{q=2}^{\infty }\frac{J_{q}^{2}(u)}{q!}\leq 
\frac{K}{n_{\ell ,d}}\Phi (u)\left( 1-\Phi (u)\right) \\
&\leq &\frac{K}{4n_{\ell ,d}},
\end{eqnarray*}%
because%
\begin{equation*}
\sum_{q=2}^{\infty }\frac{J_{q}^{2}(u)}{q!}=\func{Var}\left( 1_{[u,\infty
)}\left( T_{\ell }(x)\right) \right) =\Phi (u)\left( 1-\Phi (u)\right)
\end{equation*}%
and%
\begin{equation*}
\sup_{u}\Phi (u)\left( 1-\Phi (u)\right) =\Phi (0)(1-\Phi (0))=\frac{1}{4}%
\text{ .}
\end{equation*}
\end{proof}

\begin{lemma}
\label{lem:2} Let $M:\mathbb{R}\to \mathbb{R}$ be a measurable function such
that $\mathbb{E}[M(Z)^2]<+\infty$, $Z\sim \mathcal{N}(0,1)$ and define 
\begin{equation*}
S_\ell(M)=\int_{S^d} M(T_\ell(x))\,dx.
\end{equation*}
Then it holds that 
\begin{equation*}
\func{Var} (S_{\ell}(M)) \sim \frac{1}{n_{\ell d}}
\end{equation*}
for both, $\ell \to \infty$ and $d \to \infty$.
\end{lemma}

\begin{proof}
For $\ell \rightarrow \infty ,$ the result was given in~\cite{mrossi}.
Similarly, for $d\rightarrow \infty ,$ it holds that 
\begin{equation*}
S_{\ell }(M)=\int_{S^{d}}M(T_{\ell }(x))\,dx=\sum_{q=0}^{+\infty }\frac{%
J_{q}(M)}{q!}\underbrace{\int_{S^{d}}H_{q}(T_{\ell }(x))\,dx}_{:=h_{\ell
;q,d}}\ .
\end{equation*}%
Therefore 
\begin{equation*}
\displaylines{ \operatorname{Var}(S_\ell(M)) =
\sum_{q=2}^{+\infty}\frac{J_q(M)^2}{(q!)^2} \operatorname{Var}
(h_{\ell;q,d})\ . }
\end{equation*}%
Simple computations give 
\begin{equation*}
\func{Var}(h_{\ell ;q,d})=q!\mu _{d}\mu _{d-1}\int_{0}^{\pi }G_{\ell
;d}(\cos \vartheta )^{q}(\sin \vartheta )^{d-1}\,d\vartheta
\end{equation*}%
where $\mu _{d}$ is the Lebesgue measure of the hyperspherical surface.

If we normalize the hypersphere, we get 
\begin{equation*}
\func{Var}(h_{\ell;q,d})= q! \frac{\mu_d \mu_{d-1}}{\mu_d^2} \int_0^{\pi}
G_{\ell;d}(\cos \vartheta)^q (\sin \vartheta)^{d-1}\, d\vartheta\ .
\end{equation*}
Therefore, 
\begin{align*}
\func{Var}(S_\ell(M)) &=\sum_{q=2}^{+\infty}\frac{J_q(M)^2}{q!} \frac{\mu_d
\mu_{d-1}}{\mu_d^2} \int_0^{\pi} G_{\ell;d}(\cos \vartheta)^q (\func{sin}
\vartheta)^{d-1}\, d \vartheta \\
&\leq \frac{\mu_d \mu_{d-1}}{\mu_d^2}\int_0^{\pi} G_{\ell;d}(\cos
\vartheta)^2 (\func{sin} \vartheta)^{d-1}d\vartheta \underbrace{%
\sum_{q=2}^{+\infty}\frac{J_q(M)^2}{q!}}_{\sim\func{Var}[M(Z)]<+\infty} \\
&= \frac{\mu_d \mu_{d-1}}{\mu_d^2}\frac{\mu_d}{\mu_{d-1} n_{\ell;d}} 
\underbrace{\sum_{q=2}^{+\infty}\frac{J_q(M)^2}{q!}}_{\sim\func{Var}%
[M(Z)]<+\infty} \sim \frac{1}{n_{\ell;d}}
\end{align*}
as $\ell \to \infty$ and also as $d \to \infty$.
\end{proof}

\begin{lemma}
\label{lemmaquattro}For all $u\in \mathbb{R},$ $r\geq 0,$ $%
f:S^{2}\rightarrow \mathbb{R},$ $f\in C^{2}(S^{2}),$ we have%
\begin{equation*}
\mathcal{N}^{c}(f_{\ell };u)=\mathcal{N}^{c}(rf_{\ell };ru)\text{ .}
\end{equation*}
\end{lemma}

\begin{proof}
It is obvious that for all constant $r>0$%
\begin{equation*}
\int_{S^{2}}\left\vert \nabla ^{2}h_{\alpha ,\ell }(x)\right\vert \mathbb{%
\delta }(\left\Vert \nabla h_{\alpha ,\ell }(x)\right\Vert \mathrm{d}%
x=\int_{S^{2}}\left\vert \nabla ^{2}rh_{\alpha ,\ell }(x)\right\vert \mathbb{%
\delta }(\left\Vert \nabla rh_{\alpha ,\ell }(x)\right\Vert )\mathrm{d}x%
\text{ ,}
\end{equation*}%
whence the result follows from Kac-Rice formula.
\end{proof}

\textbf{Acknowledgements }The authors would like to thank Igor Wigman for
many insightful comments on an earlier draft of this paper. Research
supported by the European Research Council Grant 277742 \emph{Pascal. }

\bigskip

\end{document}